\documentclass[12pt]{article}
\usepackage{latexsym,amsfonts,amsmath,theorem,amssymb}
\usepackage{graphicx}
\usepackage{verbatim}
\usepackage{tikz,tikz-cd}
\scrollmode

\newtheorem{theorem}{Theorem}
\newtheorem{lemma}[theorem]{Lemma}

\newtheorem{proposition}[theorem]{Proposition}
\newtheorem{definition}[theorem]{Definition}

\newenvironment{proof}{\begin{trivlist}
    \item[\hskip\labelsep{\bf Proof.}]}{$\hfill\Box$\end{trivlist}}

{\theoremstyle{plain} \theorembodyfont{\rmfamily}
\newtheorem{remark}[theorem]{Remark}}
{\theoremstyle{plain} \theorembodyfont{\rmfamily}
\newtheorem{example}[theorem]{Example}}

\newcommand{\bsgamma}{{\boldsymbol{\gamma}}}

\newcommand{\bst}{{\boldsymbol{t}}}

\newcommand{\bsx}{{\boldsymbol{x}}}

\newcommand{\bsy}{{\boldsymbol{y}}}
\newcommand{\bsz}{{\boldsymbol{z}}}

\newcommand{\bspitch}{{\boldsymbol{\,\pitchfork}}}
\newcommand{\an}{\mathrm{A}}
\newcommand{\gc}{\star}
\newcommand{\bszero}{{\boldsymbol{0}}}

\newcommand{\rd}{\,\mathrm{d}}

\newcommand{\bbR}{\mathbb{R}}
\newcommand{\bbN}{\mathbb{N}}

\newcommand{\mask}[1]{}
\newcommand{\esup}{\operatornamewithlimits{ess\,sup}}

\newcommand{\e}{{\varepsilon}}

\newcommand{\setu}{{\mathfrak{u}}}
\newcommand{\setU}{{\mathfrak{U}}}
\newcommand{\setv}{{\mathfrak{v}}}
\newcommand{\setw}{{\mathfrak{w}}}

\title{Equivalence of Weighted Anchored and ANOVA\\
    Spaces of Functions with Mixed Smoothness\\ of Order one in 
    $L_p$}
\author{M. Gnewuch, M. Hefter, A. Hinrichs, K. Ritter, 
and G. W. Wasilkowski} 
\date{October 21, 2016}

\begin{document}
\maketitle

\begin{abstract}
We consider $\bsgamma$-weighted anchored and ANOVA spaces 
of functions with mixed first order
partial derivatives bounded in a weighted $L_p$ norm with
$1 \leq p \leq \infty$. 
The domain of the functions is $D^d$, where $D \subseteq \bbR$
is a bounded or unbounded interval. 
We provide conditions 
on the weights $\bsgamma$ that guarantee that anchored and ANOVA
spaces are equal (as sets of functions) and have equivalent
norms with equivalence constants uniformly or polynomially
bounded in $d$. 
Moreover, we discuss applications of these results
to integration and approximation of functions on $D^d$.
\end{abstract}

\section{Introduction}

This paper studies the equivalence of anchored and ANOVA spaces,
a research initiated in \cite{HeRi13} in an abstract setting
for reproducing kernel Hilbert spaces. It  
provides extensions of the results obtained in 
\cite{HeRiWa14,HS16}.
The two major differences between those papers and the 
current one are the following.
First of all, the domain of functions in the former two
papers is $[0,1]^d$, whereas we allow now for 
$D^d$ with any (bounded or unbounded) interval $D \subseteq
\bbR$. To simplify the presentation we assume that $0 = \min D$
throughout this paper.
Secondly, the standard $L_p$ norms were used 
in the former papers whereas we consider now mixed $L_p$-$\ell_q$
norms with $1 \leq p,q \leq \infty$
that are based on a probability density function
$\psi$ that is positive a.e.~on $D$. 

We now describe briefly the $\bsgamma$-weighted 
anchored and ANOVA norms and spaces, which are studied in detail
in Sections \ref{s2} and \ref{s3}. 
Consider first $d=2$ and $p=q=1$, 
and, in order to 
simplify the presentation in the introduction, a function $f : D^2
\to \bbR$ with continuous mixed partial derivatives of order one.
For a family $\bsgamma = (\gamma_{\setu})_{\setu\subseteq \{1,2\}}$
of positive reals the $\bsgamma$-weighted \emph{anchored} norm of $f$ is 
given by 
\begin{eqnarray*}
  \|f\|_{W_{\bspitch,1,1,\psi,\bsgamma}}&=&
\gamma_{\emptyset}^{-1}\,|f(0,0)|
    +\gamma_{\{1\}}^{-1} \,\int_D 
   \left|\tfrac{\partial}{\partial x_1}
      f(x_1,0)\right|\,\psi(x_1)\rd x_1\\
  &&\quad +
  \gamma_{\{2\}}^{-1} \,\int_D 
\left|\tfrac{\partial}{\partial x_2}
      f(0,x_2)\right|\,\psi(x_2)\rd x_2\\
  &&\quad +
  \gamma_{\{1,2\}}^{-1}\,\int_D    \int_D 
   \left|\tfrac{\partial^2}{
   \partial x_1 \partial x_2}f(x_1,x_2)\right|\,\psi(x_2)\rd x_2
    \,\psi(x_1)\rd x_1,
\end{eqnarray*}
and the $\bsgamma$-weighted \emph{ANOVA} norm is given by 
\begin{eqnarray*}
  \|f\|_{W_{A,1,1,\psi,\bsgamma}}&=&\gamma_{\emptyset}^{-1}\, 
   \left|\int_D \int_D 
     f(x_1,x_2)\,\psi(x_2)\rd x_2\,\psi(x_1)\rd x_1\right|\\
  &&\quad +\gamma_{\{1\}}^{-1} \,\int_D \left|\int_D 
    \tfrac{\partial}{\partial x_1}
      f(x_1,x_2)\,\psi(x_2)\rd x_2\right|\,\psi(x_1)\rd x_1\\  
  &&\quad +
  \gamma_{\{2\}}^{-1} \,\int_D   \left|\int_D 
    \tfrac{\partial}{\partial x_2}f(x_1,x_2)\,\psi(x_1)\rd x_1\right|
    \,\psi(x_2)\rd x_2\\
  &&\quad +
  \gamma_{\{1,2\}}^{-1}\,\int_D \int_D     \left|\tfrac{\partial^2}{
   \partial x_1 \partial x_2}f(x_1,x_2)\right|\,\psi(x_2)\rd x_2
    \,\psi(x_1)\rd x_1,
\end{eqnarray*}
provided that the respective integrals are finite.
For $1 < p \leq \infty$ or $1 < q \leq \infty$, 
the definition of the $\bsgamma$-weighted anchored and ANOVA 
norms are appropriately changed.

For arbitrary $d\in\bbN$
and a family $\bsgamma=(\gamma_{\setu})_{\setu\subseteq \{1,\ldots,d\}}$ 
of non-negative weights the norms of functions of $d$
variables have $2^d$ terms involving 
various mixed partial derivatives $\partial^{|\setu|}/\prod_{j\in\setu}
\partial x_j$, each weighted by $\gamma_\setu^{-1}$ for $\setu\subseteq
\{1,\dots,d\}$. In particular,
if $\gamma_\setu=0$ then the corresponding 
term involving $\partial^{|\setu|}/\prod_{j\in\setu}\partial x_j$ is 
assumed to be zero. 

Roughly speaking, the $\bsgamma$-weighted
anchored space $W_{\bspitch,p,q,\psi,\bsgamma}$ is 
the \emph{Banach space of functions}
with finite anchored norm, 
and the ANOVA space $W_{A,p,q,\psi,\bsgamma}$ 
is the \emph{Banach space of functions}
with finite ANOVA norm.
Actually there a two different ways how to rigorously
define these spaces, even for $D=[0,1]$ and $\psi=1$, which is
most frequently studied in the literature.
The first approach is to consider weak derivatives, but then
the problem arises how to define the section (for the anchored norm) or 
the integral (for the ANOVA norm) of a weak derivative on a set of 
measure zero; we refer to \cite{HeRiWa14} for a solution in the
case $D=[0,1]$, $p=q$, and $\psi=1$. The second
approach, which is presented here and which is well suited
for the application of interpolation theory,
is based on smoothing and superposition of
functions with suitable integrability properties.
This approach, however, does not immediately yield
an intrinsic characterization of the whole space via
differentiability properties.

We only state here 
that $p$ and $\psi$ have to satisfy certain integrability
conditions, see \eqref{eq1} and \eqref{eq2},
for the spaces to be well defined, 
and that the spaces are continuously embedded into the space
of continuous functions on $D^d$.
The integrability conditions reveal, in particular, that
$m_\psi< \infty$ and $\kappa_\psi < \infty$ for
\[
  m_\psi\,=\,\int_D y\,\psi(y)\rd y
\] 
and
\[
  \kappa_\psi \,=\,\esup_{t\in D}\frac{\int_{D \cap (t,\infty)}
\psi(y)\rd y}{\psi(t)} 
\]
are necessary for the anchored and the ANOVA space to be well 
defined for any $p$. 

The main aim of this paper is to compare the spaces
$W_{\bspitch,p,q,\psi,\bsgamma}$ and $W_{A,p,q,\psi,\bsgamma}$ 
and their norms, see Section \ref{s4}. It turns out that
$W_{\bspitch,p,q,\psi,\bsgamma}=W_{A,p,q,\psi,\bsgamma}$ if and only if 
$\bsgamma$ satisfies the condition
\[
  \gamma_\setw\,>\,0\quad\mbox{implies that}\quad
  \gamma_\setu\,>\,0\mbox{\ for all\ }\setu\,\subseteq\,\setw, 
\] 
which is assumed to hold for the rest of the introduction.
Let 
\[
  \imath_{p,q,\psi,\bsgamma}:     
W_{\an,p,q,\psi,\bsgamma}\hookrightarrow W_{\bspitch,p,q,\psi,\bsgamma}
\]
denote the embedding operator from the ANOVA space into the
anchored space, which, together with its inverse, is continuous due to the
closed graph theorem. We show that both, 
$W_{\bspitch,p,q,\psi,\bsgamma}$ and $W_{A,p,q,\psi,\bsgamma}$
are isometrically isomorphic to a $\bsgamma$-weighted $\ell_q$
sum of spaces $L_{p,\psi}(D^\setu)$. The latter consist of
equivalence classes of real-valued functions on $D^\setu$, where
$\setu \subseteq D^d$, with norms given by
\[
    \|g\|_{L_{p,\psi}(D^\setu)}\,=\,\left(\int_{D^\setu}|g(\bsx)|^p
    \prod_{j=1}^d\psi(x_j)\rd\bsx\right)^{1/p}
\]
for $1 \leq p < \infty$, with the usual modification for $p=\infty$.
We employ these isometric isomorphisms to show that
\[
\|\imath_{p,q,\psi,\bsgamma}\| \,=\,
\|\imath^{-1}_{p,q,\psi,\bsgamma}\|.
\]
Moreover, we provide explicit expressions for the norm
$\|\imath_{p,q,\psi,\bsgamma}\|$ of the embedding in terms of 
$\bsgamma$, $m_\psi$, and $\kappa_\psi$ in the four extremal cases
corresponding to $p,q \in \{1,\infty\}$. In all other cases
we use complex interpolation theory to get upper bounds 
for $\|\imath_{p,q,\psi,\bsgamma}\|$.

Observe that $\|\imath_{p,q,\psi,\bsgamma}\|$ depends on the 
number $d$ of variables only via the family $\bsgamma$ of weights.
In Section \ref{s5} we study, for a number of different classes of 
weights, when the norms $\|\imath_{p,q,\psi,\bsgamma}\|$ of the 
embeddings are uniformly bounded in $d$. 
If this holds, then we say that the $\bsgamma$-weighted anchored 
and ANOVA spaces 
are \emph{equivalent independently of the dimension}. 

The paper is organized as follows. Section \ref{s2} provides basic facts 
used in the paper. The $\bsgamma$-weighted anchored and ANOVA
spaces and norms are discussed in Section \ref{s3}. The main
results are stated in Section \ref{s4}. 
Applications of the results to special classes of weights can be found
in Section \ref{s5}. 
Applications of the results from Sections \ref{s4}
and \ref{s5} to integration and approximation of functions
are discussed in Section \ref{s6}.

\section{Notation and Basic Facts}\label{s2}
Let 
$D \subseteq [0,\infty)$ be an interval with $0 \in D$, i.e.,
\[
  D\,=\,[0,T)\mbox{\ for some\ }T\in(0,\infty]\quad\mbox{or}\quad
  D\,=\,[0,T]\mbox{\ for some\ }T\in(0,\infty),
\]
and let 
\[
   \psi:D\to [0,\infty)
\]
be a probability density function that is positive almost
everywhere.

Let $d$ be a positive integer. 
In what follows we will use $\setu,\setv,$ and $\setw$ to denote subsets
of $[1:d]$, where 
\[
  [1:d]\,=\,\{1,\dots,d\},
\]
and we will denote the complement of $\setu$ in $[1:d]$ by 
$\setu^c$.
We will also use $\bst,\bsx,\bsy,$ and $\bsz$ to denote points from $D^d$,
and we will often use the following notation
\[
   [\bsx_\setu;\bst_{\setu^c}]\,=\,(y_1,\dots,y_d)\quad\mbox{with}\quad
   y_j\,=\,\left\{\begin{array}{ll} x_j&\mbox{if\ }j\in\setu,\\
         t_j&\mbox{if\ }j\in \setu^c.\end{array}\right.
\]
For $\setu \neq \emptyset$
we will write $D^\setu$ to denote the set of points 
$\bsx_\setu=(x_j)_{j\in\setu}$ with $x_j\in D$. To simplify the notation 
we will often write $\bsx_\setu;\bst_{\setu^c}$ instead of 
$[\bsx_\setu;\bst_{\setu^c}]$. 

In the sequel, let $\setu \neq \emptyset$.
We will often consider real-valued functions on $D^d$ that only
depend on the variables with indices from $\setu$. To simplify the
notation we will identify any such function with a function 
on $D^\setu$ in the canonical way; vice versa, functions of
the latter kind are identified with functions on $D^d$.

In the sequel,
the integrability index $p$ satisfies $1 \leq p \leq \infty$, and
$p^\prime$ denotes its conjugate index given by $1/p+1/p^\prime=1$.
Furthermore, 
let 
\[
\psi_\setu(\bst_\setu)\,=\,\prod_{j\in\setu}\psi(t_j).
\]
Let $L_{p,\psi}(D^\setu)$ be the Banach space of 
functions on $D^\setu$ with the corresponding weighted $L_p$ norm.
For $1 \leq p < \infty$ this norm is given by
\[
  \|g\|_{L_{p,\psi}(D^\setu)}\,=\,\left(\int_{D^\setu}|g(\bst_\setu)|^p\,
  \psi_\setu(\bst_\setu)\rd\bst_\setu\right)^{1/p}.
\]
For $p=\infty$ we have the usual modification
\[
  \|g\|_{L_{\infty,\psi}(D^\setu)}\,=\,\|g\|_{L_\infty(D^\setu)}\,=\,
  \esup_{\bst_\setu\in D^\setu}|g(\bst_\setu)|,
\]
which does not depend on $\psi$.
The space of locally $p$-integrable functions on  $D^{\setu}$ is given by
\[
L_{p}^{\mathrm{loc}}(D^{\setu}) = 
\{f:D^\setu \to \bbR \,:\, f|_{[{\bf 0}, {\bsx})} \in 
L_p([{\bf 0}, {\bsx}) )
\hspace{1ex}\text{for all $\bsx \in D^\setu$}\}.
\]
Here $[{\bf 0}, {\bsx})$ denotes the half-open interval
in $D^\setu$ with lower left corner ${\bf 0}$ and upper right corner 
$\bsx$.  

By $1_{[0,x)}$ we denote the indicator function of the interval $
\left[0,x\right) \subseteq \bbR$. Consider 
\[
  K_\bspitch(x,t)\,=\, 1_{[0,x)} (t) 
\quad\mbox{and}\quad
  K_\an(x,t)\,=\, 1_{[0,x)}(t)-\overline{\psi}(t),
\]
where $x,t \in D$ and
\[
  \overline{\psi}(t)\,=\,\int_t^T\psi(y)\rd y\,=\,
  \int_DK_\bspitch(y,t)\, \psi(y)  \rd y. 
\]
Define 
\[
   K_{\bspitch,\setu}(\bsx_\setu,\bst_\setu)\,=\,
    \prod_{j\in\setu}K_\bspitch(x_j,t_j) \quad\mbox{and}\quad
  K_{\an,\setu}(\bsx_\setu,\bst_\setu)\,=\,
  \prod_{j\in\setu}K_\an(x_j,t_j).
\]
Both, $K_{\bspitch,\setu}$ and $K_{\an,\setu}$, will
be used as integral kernels in the construction of the anchored
and the ANOVA spaces. At first we study the basic integrability
properties.

\begin{lemma}\label{l1}\hfill
\begin{itemize}
\item[{\rm (i)}]
We have
\[
  \int_{D^\setu} |g_\setu(\bst_\setu)|\,
    K_{\bspitch,\setu}(\bsx_\setu,\bst_\setu)\rd\bst_\setu
< \infty
\]
for every $g_\setu \in L_{p,\psi}(D^\setu)$ and every $\bsx_\setu \in
D^\setu$ if and only if 
\begin{equation}\label{eq1}
\left.
\begin{aligned}
&p = \infty \quad \mbox{or}\\
&p < \infty \quad\mbox{and}\quad
\psi^{-1/p} \in L_{p^\prime}^{\mathrm{loc}}(D).
\end{aligned}\quad
\right\}
\end{equation}
\item[{\rm (ii)}]
We have
\[
  \int_{D^\setu} |g_\setu(\bst_\setu)|\,
    |K_{\an,\setu}(\bsx_\setu,\bst_\setu)|\rd\bst_\setu
< \infty
\]
for every $g_\setu \in L_{p,\psi}(D^\setu)$ and every $\bsx_\setu \in
D^\setu$ if and only if 
\begin{equation}\label{eq2}
\left.
\begin{aligned}
& p = \infty \quad\mbox{and}\quad \overline{\psi} \in L_1(D)
\quad \mbox{or} \\
& p < \infty 
\quad\mbox{and}\quad
\psi^{-1/p} \in L_{p^\prime}^{\mathrm{loc}}(D)
\quad\mbox{and}\quad
\overline{\psi} \cdot \psi^{-1/p} \in L_{p^\prime}(D).
\end{aligned}\quad
\right\}
\end{equation}
\end{itemize}
\end{lemma}

\begin{proof}
For $\gc\in\{\bspitch,\an\}$, $g_\setu\in L_{p,\psi}(D^\setu)$,
and $\bsx_\setu\in D^\setu$ we have
\begin{align*}
\int_{D^\setu}
|g_\setu(\bst_\setu)|\,
|K_{\gc,\setu}(\bsx_\setu,\bst_\setu)|
\rd\bst_\setu
\,=\,\int_{D^\setu}
|g_\setu(\bst_\setu)|\,
\frac{|K_{\gc,\setu}(\bsx_\setu,\bst_\setu)|}
{\psi_\setu(\bst_\setu)}\,
\psi_\setu(\bst_\setu)
\rd\bst_\setu.
\end{align*}
Lemma~\ref{prop1} from the Appendix shows that
\begin{align*}
\int_{D^\setu}
|g_\setu(\bst_\setu)|\,
|K_{\gc,\setu}(\bsx_\setu,\bst_\setu)|
\rd\bst_\setu
\,<\,\infty
\quad \text{for all } g_\setu\in L_{p,\psi}(D^\setu)
\text{ and } \bsx_\setu\in D^\setu
\end{align*}
if and only if
\begin{align*}
\frac{K_{\gc,\setu}(\bsx_\setu,\cdot_\setu)}{\psi_\setu}
\in L_{p^\prime,\psi}(D^\setu) 
\quad \text{for all } \bsx_\setu\in D^\setu,
\end{align*}
which is equivalent to
\begin{align}\label{e50}
\frac{K_{\gc}(x,\cdot)}{\psi}
\in L_{p^\prime,\psi}(D)
\quad \text{for all } x\in D.
\end{align}
Furthermore, we use $K_{\an}(0,\cdot) = - \overline{\psi}$ to
conclude that \eqref{e50} with $\gc=\an$ is equivalent to
\eqref{e50} with $\gc = \bspitch$ and
\begin{equation}\label{e51}
\overline{\psi}/\psi \in L_{p^\prime,\psi}(D).
\end{equation}

Consider $\gc=\bspitch$. Then we have \eqref{e50} if and only if
\[
\psi^{-1} \in L_{p^\prime,\psi}^{\mathrm{loc}}(D),
\]
which yields the claim in (i).

Consider $\gc=\an$. To establish the claim in (ii) 
one easily verifies 
that \eqref{e51} is equivalent to $\overline{\psi} \in
L_1(D)$ if $p=\infty$ and to $\overline{\psi}\cdot \psi^{-1/p} \in
L_{p^\prime}(D)$ if $p < \infty$.
\end{proof}

We comment on the conditions \eqref{eq1} and \eqref{eq2}.

\begin{remark}\label{r3}
Obviously, \eqref{eq2} implies \eqref{eq1},
and Lemma \ref{l1} yields the following monotonicity
property. If one of these conditions is satisfied for
$p=p_1$ and $\psi$, then it also holds for the same
density $\psi$ and every $p>p_1$.
Given \eqref{eq2}, we obtain 
$\overline{\psi} \in L_1(D)$, and therefore
\begin{align}\label{mean}
   m_\psi \,=\, 
\int_D\overline{\psi}(t)\rd t \,=\,
   \int_Dy\,\psi(y)\rd y \in (0, \infty).
\end{align}
\end{remark}

\begin{remark}\label{r1}
Consider the assumption \eqref{eq2} in the case $p<\infty$.
If $D$ is compact,
then $L_{p^\prime}^{\mathrm{loc}}(D) = L_{p^\prime}(D)$ 
and $\psi^{-1/p} \in
L_{p^\prime}(D)$ implies $\overline{\psi} \cdot \psi^{-1/p} \in
L_{p^\prime}(D)$.
Therefore we have equivalence of \eqref{eq1} and \eqref{eq2}
for compact sets $D$.
If $D$ is not compact,
then $\overline{\psi} \cdot \psi^{-1/p} \in L_{p^\prime}(D)$ implies
$\psi^{-1/p} \in L^{\mathrm{loc}}_{p^\prime}(D)$. 
\end{remark}

\begin{example}\label{ex1}
Consider the case of a bounded interval $D$ with
$T=1$ for simplicity. Let
\[
\psi(t) \,=\, (\alpha+1) \cdot (1-t)^\alpha 
\]
for $\alpha > -1$, so that
\[
\overline{\psi}(t) \,=\, (1-t)^{\alpha+1}.
\]
The following facts are easily verified with the help of Remark \ref{r1}.
If $D=[0,1)$, then \eqref{eq2} holds true for
every $p$.
If $D=[0,1]$, then \eqref{eq2} holds true if and only if
$p > \alpha+1$ or $p=1$ and $\alpha = 0$.
\end{example}

\begin{example}\label{ex2}
Consider the unbounded interval $D=[0,\infty)$. Let
\[
\psi(t) \,=\, (\alpha-1) \cdot (1+t)^{-\alpha}
\]
for $\alpha > 1$, so that
\[
\overline{\psi}(t) \,=\, (1+t)^{1-\alpha}.
\]
Clearly, we have \eqref{eq1} for every $p$.
With the help of Remark \ref{r1} we easily verify that
\eqref{eq2} holds true if and only if $\alpha > 2$ and
$p > 1+1/(\alpha-2)$.
\end{example}

\begin{example}\label{ex3}
Consider again $D=[0,\infty)$. Let
\[
\psi(t) \,=\, c \cdot \exp \left(-b\, t^a\right),
\]
where $a,b > 0$ and 
$c = 1/ \int_0^\infty \exp\left(-b \, t^a\right) \rd t$.
Clearly, we have \eqref{eq1} for every $p$.
We claim that \eqref{eq2} holds true if and only if $a \geq1$ or
$p>1$.
Note that
\begin{align*}
\lim_{t\to\infty} \frac{\overline\psi(t)}{t^{1-a}\exp(-bt^a)}
\,=\, c/(ab) \in (0,\infty),
\end{align*}
which follows from L'H\^{o}pital's rule. Once more, it remains
to apply Remark \ref{r1}.
\end{example}

For the rest of this section 
let
$\gc \in \{\bspitch,\an\}$ and assume that
\eqref{eq1} is satisfied if $\gc=\bspitch$
and that \eqref{eq2} is satisfied if $\gc=\an$.
For $g_\setu \in L_{p,\psi}(D^\setu)$ we put
\[
  T_{\gc,\setu}(g_\setu)\,=\,\int_{D^\setu} g_\setu(\bst_\setu)\,
    K_{\gc,\setu}(\cdot_\setu,\bst_\setu)\rd\bst_\setu,
\]
which is well defined due to Lemma \ref{l1}.

\begin{remark}\label{r4}
By $f^{(\setu)}$ we mean
   $f^{(\setu)}\,=\,
   \prod_{j\in\setu}\frac\partial{\partial x_j} f$,
where $\frac\partial{\partial x_j} f$ denotes the distributional 
derivative of $f$ with respect to $x_j$. 
Lemma \ref{l1} and Remark \ref{r3} imply that $L_{p,\psi}(D^\setu)
\subseteq L^{\mathrm{loc}}_1(D^\setu)$. It follows that 
$f = T_{\gc,\setu}(g_\setu)$ with any $g_\setu \in
L_{p,\psi}(D^\setu)$ has a weak derivative $f^{(\setu)}$ and
$f^{(\setu)} = g_\setu$.
\end{remark}

On the space $C(D^d)$ of continuous real-valued functions on $D^d$
we consider the topology of uniform convergence on compact subsets.

\begin{lemma}\label{l2}
We have $T_{\gc,\setu} (g_\setu) \in C(D^d)$ for every $g_\setu \in
L_{p,\psi}(D^\setu)$, and
the mapping 
\[
T_{\gc,\setu}: L_{p,\psi}(D^\setu) \to C(D^d)
\]
is linear, continuous, and one-to-one.
\end{lemma}

\begin{proof}
We obtain $T_{\gc,\setu} (g_\setu) \in C(D^d)$ from
\[
\left| 
(T_{\gc,\setu} (g_\setu)) (\bsx_\setu) -
(T_{\gc,\setu} (g_\setu)) (\bsy_\setu) 
\right| 
\leq
\int_{D^\setu} |g_\setu(\bst_\setu)|\, 
|K_{\bspitch,\setu}(\bsx_\setu,\bst_\setu) -
K_{\bspitch,\setu}(\bsy_\setu,\bst_\setu)|
\rd\bst_\setu 
\]
and $g_\setu \in L_{1}^{\mathrm{loc}}(D^\setu)$, see Remark \ref{r4}.
Linearity of $T_{\gc,\setu}$ obviously holds. To prove continuity
of $T_{\gc,\setu}$ it suffices to show that
\[
\sup_{x \in [0,y]} \left\| \frac{K_\gc(x,\cdot)}{\psi}
\right\|_{L_{p^\prime,\psi}(D)}  < \infty
\]
for every $y \in D$, cf.\ the proof of Lemma \ref{l1}.
The latter property holds true due to \eqref{eq1} if $\gc =
\bspitch$ or \eqref{eq2} if $\gc=\an$.
It remains to show that $T_{\gc,\setu}$ is one-to-one. 
For this purpose consider $g_\setu$ such that $T_{\gc,\setu}(g_\setu)
=0$. 
Remark \ref{r4} yields $g_\setu=0$,
which completes the proof. 
\end{proof}

\section{Anchored and ANOVA Spaces}\label{s3}

As previously, let $\gc \in \{\bspitch,\an\}$, and assume that
\eqref{eq1} is satisfied if $\gc=\bspitch$
and that \eqref{eq2} is satisfied if $\gc=\an$.

For $\setu\not=\emptyset$ we define the spaces
\begin{align*}
F_{\gc,p,\psi,\setu}
\,=\,
T_{\gc,\setu}(L_{p,\psi}(D^\setu)).
\end{align*}
For $\setu = \emptyset$ and $c \in \bbR$
we put $L_p(D^\setu) = L_{p,\psi}(D^\setu) = \bbR$ with 
$\|c\|_{L_{p,\psi}(D^\setu)} = |c|$, and
$T_{\gc,\setu}(c)$ denotes the constant function with value $c$.
Hence $F_{\gc,p,\psi,\emptyset} = 
T_{\gc,\emptyset}(L_{p,\psi}(D^\emptyset))$ is the space of constant
functions on $D^d$.

Let $\setu\subseteq [1:d]$ and $f\in F_{\gc,p,\psi,\setu}$.
Note that $f(\bsx)$ depends on $\bsx$ only through $\bsx_\setu$.
Furthermore, 
\begin{align*}
f(\bsx)
\,=\,
0
\quad
\text{if}
\quad
x_j
\,=\,
0
\quad
\text{for any}
\quad
j\in\setu,
\end{align*}
if $\gc=\bspitch$,
and
\begin{align*}
\int_D
f(\bsx)
\,\psi(x_j)\rd x_j
\,=\,
0
\quad
\text{if}
\quad
j\in\setu,
\end{align*}
if $\gc=\an$, for any $f\in F_{\gc,p,\psi,\setu}$.
This leads directly to the following lemma;
see \cite[Thm.~2.1]{KSWW09b}
for a general result on decomposition of functions.

\begin{lemma}\label{decompwelldef}
Let $f_{\gc,\setu},\tilde f_{\gc,\setu} \in F_{\gc,p,\psi,\setu}$
for $\setu\subseteq [1:d]$.
Then we have
\begin{align*}
\sum_{\setu\subseteq [1:d]} f_{\gc,\setu}
=\sum_{\setu\subseteq [1:d]} \tilde f_{\gc,\setu}
\qquad
\text{if and only if}
\qquad
f_{\gc,\setu}
=\tilde f_{\gc,\setu}
\quad
\text{for all}
\quad
\setu\subseteq [1:d].
\end{align*}
\end{lemma}

We identify the elements $(f_{\gc,\setu})_{\setu \in [1:d]}$ of the 
direct sum
$\bigoplus_{\setu\in [1:d]} F_{\gc,p,\psi,\setu}$
with the continuous functions 
$\sum_{\setu\in [1:d]} f_{\gc,\setu}$ on $D^d$,
which is possible due to Lemma~\ref{decompwelldef}. 
The representation
\begin{equation}\label{g20}
f \,=\, \sum_{\setu\subseteq [1:d]} f_{\gc,\setu}
\,=\, \sum_{\setu\subseteq [1:d]} T_{\gc,\setu} (g_{\gc,\setu})
\end{equation}
with $g_{\gc,\setu} \in L_{p,\psi}(D^\setu)$ and
$f_{\gc,\setu} = T_{\gc,\setu} (g_{\setu})$ is called
the \emph{anchored decomposition} of $f$ in the case $\gc = \bspitch$
and the \emph{ANOVA decomposition} of $f$ in the case 
$\gc = \an$. We provide an explicit relation between the
corresponding functions $g_{\gc,\setu}$ in these two
decompositions. Put 
\[
\overline{\psi}_\setw (\bst_\setw) = 
\prod_{j\in\setw}\overline{\psi}(t_j).
\]
For convenience of notation we set
$\int_{D^\setw} f(\bst_\setw) \rd \bst_\setw = f$ for $\setw =
\emptyset$.

\begin{lemma}\label{lem:2}
Assume that \eqref{eq2} is satisfied. 
If $\setw\subseteq \setu^c$ and
$g_{\gc,\setu \cup \setw }\in L_{p,\psi}(D^{\setu \cup \setw})$,
then
\begin{equation}\label{gl1}
\left.
\begin{aligned}
&g_{\gc,\setu\cup\setw}(\bsx_\setu;\cdot_\setw)\cdot
\overline{\psi}_\setw (\cdot_\setw) \in L_1(D^\setw)
\ \text{for a.e.\ $\bsx_\setu$ if $\setu \neq \emptyset$}
\quad \mbox{and}\\
&
\int_{D^\setw}
g_{\gc,\setu\cup\setw}(\cdot_\setu;\bst_\setw)\cdot
\overline{\psi}_\setw (\bst_\setw)\rd\bst_\setw
\in L_{p,\psi}(D^\setu).
\end{aligned}\quad
\right\}
\end{equation}
Moreover, let $g_{\bspitch,\setu}, g_{A,\setu} \in
L_{p,\psi}(D^\setu)$ for every $\setu$. Then
\begin{align}\label{gl2}
\sum_{\setu\subseteq[1:d]}
T_{\bspitch,\setu}(g_{\bspitch,\setu})
\,=\,
\sum_{\setu\subseteq[1:d]}
T_{A,\setu}(g_{A,\setu})
\end{align}
if and only if
\begin{align}\label{11}
g_{\bspitch,\setu}
\,=\,
\sum_{\setw\subseteq\setu^c}
(-1)^{|\setw|}\,
\int_{D^\setw}g_{A,\setu\cup\setw}(\cdot_\setu;\bst_\setw)\,
\overline{\psi}_\setw (\bst_\setw)\rd\bst_\setw
\end{align}
and
\begin{align}\label{22}
g_{A,\setu}
\,=\,
\sum_{\setw\subseteq\setu^c}
\int_{D^\setw}g_{\bspitch,\setu\cup\setw}(\cdot_\setu;\bst_\setw)\,
\overline{\psi}_\setw (\bst_\setw)\rd\bst_\setw.
\end{align}
\end{lemma}

\begin{proof}
In the proof of \eqref{gl1} we consider the non-trivial case 
$\setw \neq \emptyset$.
Let $h \in L_{p^\prime,\psi}(D^{\setu})$, and put
\[
\tilde{h} (\bsx_\setu;\bst_\setw) =
g_{\gc,\setu\cup\setw}(\bsx_\setu;\bst_\setw)
\, h(\bsx_\setu).
\]
It follows that 
$\tilde{h}(\cdot_\setu;\bst_\setw) \,
\psi_\setu(\cdot_\setu) \in L_1(D^\setu)$ 
for a.e.\ $\bst_\setw$ and 
\[
\int_{D^\setu} |\tilde{h} (\bsx_\setu,\cdot_{\setw})|
\, \psi_\setu (\bsx_\setu) \rd \bsx_\setu 
\in L_{p,\psi}(D^\setw).
\]
Consider the case $p < \infty$, and
put $c = \|\overline{\psi} \cdot \psi^{-1/p}\|_{L_{p^\prime}(D)} <
\infty$, see \eqref{eq2}. We obtain 
\begin{align*}
&\int_{D^\setu} \!
\int_{D^\setw}
|g_{\gc,\setu\cup\setw}(\bsx_\setu;\bst_\setw)|\,
\overline{\psi}_\setw(\bst_\setw)\rd\bst_\setw
\cdot |h(\bsx_\setu)| \, \psi_\setu (\bsx_\setu) \rd \bsx_\setu \\
&\qquad = \int_{D^\setw} \!  \int_{D^\setu}
|\tilde{h}(\bsx_\setu;\bst_\setw)|\, \psi_\setu (\bsx_\setu) 
\rd\bsx_\setu \cdot
\psi_\setw^{1/p} (\bst_\setw) \cdot
\overline{\psi}_\setw(\bst_\setw) / \psi_\setw^{1/p} (\bst_\setw)
\rd \bst_\setw \\
&\qquad \leq 
\left\| \int_{D^\setu}
|\tilde{h}(\bsx_\setu;\cdot_\setw)|\, \psi_\setu (\bsx_\setu) 
\rd\bsx_\setu \cdot \psi_\setw^{1/p} (\cdot_\setw)
\right\|_{L_p(D^\setw)} \cdot c^{|\setw|} \\
&\qquad < \infty.
\end{align*}
Lemma~\ref{prop1} from the Appendix shows that
\[
\int_{D^\setw}
|g_{\gc,\setu\cup\setw}(\cdot_\setu;\bst_\setw)|\,
\overline{\psi}_\setw(\bst_\setw)\rd\bst_\setw
\in L_{p,\psi}(D^\setu),
\]
and hereby we get \eqref{gl1}.
For $p=\infty$ the proof of \eqref{gl1} is straightforward
and thus omitted.

Now we prove the equivalence of \eqref{gl2} and \eqref{11}. Take any 
$f_{\an,\setu}=T_{\an,\setu}(g_{\an,\setu})\in F_{\an,p,\psi,\setu}$. Then 
\begin{align*}
 f_{\an,\setu}(\bsx) &= \int_{D^\setu} g_{\an,\setu}(\bst_\setu)\,
    K_{\an,\setu} (\bsx_\setu,\bst_\setu) \rd\bst_\setu\\
  &=\int_{D^\setu}g_{A,\setu}(\bst_\setu)\,\sum_{\setv\subseteq\setu}
     K_{\bspitch,\setv} (\bsx_\setv,\bst_\setv) 
(-1)^{|\setu|-|\setv|} \overline{\psi}_{\setu \setminus \setv} 
(\bst_{\setu \setminus \setv})    \rd\bst_\setu\\
  &= \sum_{\setv\subseteq\setu} (-1)^{|\setu|-|\setv|} 
	      \int_{D^\setv} K_{\bspitch,\setv} (\bsx_\setv,\bst_\setv)  
				\left(\int_{D^{\setu\setminus\setv}} 
g_{\an,\setu}(\bst_\setv;\bst_{\setu\setminus\setv})\,
      \overline{\psi}_{\setu \setminus \setv} 
(\bst_{\setu \setminus \setv})    
\rd\bst_{\setu \setminus\setv} \right) \rd\bst_\setv\\
  &= \sum_{\setv\subseteq\setu}T_{\bspitch,\setv}(h_{\setu,\setv}),
\end{align*}
where
\[
  h_{\setu,\setv}(\bst_\setv)\,=\,(-1)^{|\setu|-|\setv|}\,
  \int_{D^{\setu\setminus\setv}} 
g_{\an,\setu}(\bst_\setv;\bst_{\setu\setminus\setv})\,
      \overline{\psi}_{\setu \setminus \setv} 
(\bst_{\setu \setminus \setv})    \rd\bst_{\setu \setminus\setv}.
\]
Note that the appearing integrals are well defined by \eqref{gl1}.
Furthermore, we get 
\begin{align*}
 \sum_{\setu\subseteq[1:d]}T_{A,\setu}(g_{A,\setu})
 &=
 \sum_{\setu\subseteq[1:d]}\sum_{\setv\subseteq\setu} 
  T_{\bspitch,\setv}(h_{\setu,\setv}) \,=\, 
  \sum_{\setv\subseteq[1:d]}\sum_{\setw\subseteq\setv^c}
T_{\bspitch,\setv}
   (h_{\setv\cup\setw,\setv})\\
 &= 
  \sum_{\setv\subseteq[1:d]}
T_{\bspitch,\setv}\left(\sum_{\setw\subseteq
   \setv^c}h_{\setv\cup\setw,\setv}\right).
\end{align*}
Given \eqref{gl2},
Lemma \ref{l2} and Lemma \ref{decompwelldef} imply \eqref{11}.
Conversely, if $g_{\bspitch,\setu}$ is given by \eqref{11},
then we obtain \eqref{gl2}. The equivalence of \eqref{gl2} and \eqref{22} 
can be established in the same way.
\end{proof}

\begin{example}\label{ex4}
Assume that \eqref{eq2} is satisfied, and let $\eta_\setu \in \bbR$ 
as well as 
\[
g_\setu(\bsx) \,=\, \prod_{j \in \setu} g(x_j),
\]
where $g \in L_{p,\psi}(D)$. Moreover, let
\begin{equation}\label{g22}
f_\gc = \sum_{\setu \subseteq [1:d]} \eta_\setu T_{\gc,\setu}(g_\setu).
\end{equation}
The right-hand side in \eqref{g22}
is the anchored or ANOVA decomposition, respectively, of
$f_\gc$, and its components $\eta_\setu T_{\gc,\setu}(g_\setu)$ are
of tensor product form.
We use Lemma \ref{lem:2} 
to compute the ANOVA decomposition of $f_\bspitch$ and the anchored
decomposition of $f_\an$. 
Since $\overline{\psi}(t) = K_\bspitch(x,t) -
K_\an(x,t)$, we get $g \cdot \overline{\psi} \in L_1(D)$ from
Lemma \ref{l1}. Put 
\[
c \,=\, \int_D g(t) \cdot \overline{\psi}(t)\rd t.
\]
We obtain 
\[
g_{\an,\setu} \,=\, 
\sum_{\setw\subseteq\setu^c} \eta_{\setu \cup \setw}\, c^{|\setw|} 
\cdot g_{\setu},
\]
and therefore
\[
f_\bspitch \,=\,
\sum_{\setu \subseteq [1:d]} 
\sum_{\setw\subseteq\setu^c}
\eta_{\setu\cup\setw} \cdot c^{|\setw|} \cdot T_{\an,\setu}
(g_\setu)
\]
is the ANOVA decomposition of $f_\bspitch$. In the same way
we obtain
\[
f_\an =
\sum_{\setu \subseteq [1:d]} 
\sum_{\setw\subseteq\setu^c}
\eta_{\setu\cup\setw} \cdot (-c)^{|\setw|} \cdot T_{\bspitch,\setu}
(g_\setu)
\]
as the anchored decomposition of $f_\an$.
In both cases, the components of the new decomposition are
again of tensor product form.
\end{example}

For $\setu \neq \emptyset$ we endow the spaces $F_{\gc,p,\psi,\setu}$
with the norms
\begin{align*}
\|T_{\gc,\setu}(g_\setu)\|_{F_{\gc,p,\psi,\setu}}
\,=\,
\|g_\setu\|_{L_{p,\psi}(D^\setu)},
\end{align*}
which are well defined due to Lemma~\ref{l2}.
Moreover, the space $F_{\gc,p,\psi,\emptyset}$ of constant
functions is equipped with its natural norm.

Consider a family $\bsgamma=(\gamma_\setu)_{\setu\subseteq[1:d]}$ of 
non-negative numbers, called \emph{weights}, and put
\begin{align*}
\setU_\bsgamma
\,=\,
\{\setu\subseteq [1:d]\,:\,\gamma_\setu>0\}.
\end{align*}
Henceforth we assume that $\setU_\bsgamma \neq \emptyset$.
{Let $1 \leq q \leq \infty$.
For any family $(a_\setu)_\setu = (a_\setu)_{\setu \in
\setU_\bsgamma}$ of real numbers we put 
\[
\left|(a_\setu)_\setu\right|_q = 
\left( \sum_{\setu\in\setU_\bsgamma} |a_\setu|^q \right)^{1/q}
\]
if $q < \infty$ and
\[
\left|(a_\setu)_\setu\right|_\infty = 
\max_{\setu\in\setU_\bsgamma} |a_\setu|
\]
if $q=\infty$.
}

We endow the function spaces 
\[
W_{\gc,p,q,\psi,\bsgamma} = 
\bigoplus_{\setu\in\setU_\bsgamma} F_{\gc,p,\psi,\setu}
\]
with the following $\bsgamma$-weighted 
anchored and ANOVA norms. 
For $f \in W_{\gc,p,q,\psi,\bsgamma}$
given by \eqref{g20}, i.e., $f_{\gc,\setu}=0$ for $u \neq
\setU_\bsgamma$, we put
\[
\|f\|_{W_{\gc,p,q,\psi,\bsgamma}}
\,=\,
\left| \left( 
\|f_{\gc,\setu}\|_{F_{\gc,p,\psi,\setu}}/\gamma_\setu
\right)_\setu \right|_q
\,=\,
\left| \left( 
\|g_{\gc,\setu}\|_{L_{p,\psi}(D^\setu)}/\gamma_\setu
\right)_\setu \right|_q.
\]
Note that $\|\cdot\|_{W_{\gc,\infty,q,\psi,\bsgamma}}$
does not depend on $\psi$.
Clearly, all spaces $W_{\gc,p,q,\psi,\bsgamma}$,
equipped
with the corresponding norm $\|\cdot\|_{W_{\gc,p,q,\psi,\bsgamma}}$,
are Banach spaces,
which are continuously embedded into $C(D^d)$, see Lemma \ref{l2}.
In particular,
every point evaluation $f \mapsto f(\bsx)$
on $W_{\gc,p,q,\psi,\bsgamma}$ is continuous.
We refer to $W_{\bspitch,p,q,\psi,\bsgamma}$ and
$W_{A,p,q,\psi,\bsgamma}$ as 
$\bsgamma$-weighted \emph{anchored} and \emph{ANOVA spaces}, 
respectively. 

\begin{remark}\label{r5}
In the particular case $D=[0,1]$, $\psi = 1$, and $p=q$,
the spaces 
$W_{\gc,p,p,1,\bsgamma}$ have been studied in \cite{HeRiWa14}.
Actually, these spaces were introduced via 
differentiability properties of their elements $f \in L_p(D^d)$,
namely, the existence of all weak derivatives $f^{(\setu)}$ in $L_p(D^d)$
together with
\[
\gamma_\setu = 0 \quad \Rightarrow \quad
f^{(\setu)} (\cdot_\setu;\bszero_{\setu^c}) = 0
\]
in the case $\gc=\bspitch$ and
\[
\gamma_\setu = 0 \quad \Rightarrow \quad
\int_{D^{\setu^c}} f^{(\setu)} (\cdot_\setu;\bst_{\setu^c})\rd
\bst_{\setu^c} = 0
\]
in the case $\gc=\an$. See \cite[Lem.~3, Rem.~5]{HeRiWa14} for
the definition of $f^{(\setu)} (\cdot_\setu;\bst_{\setu^c})$
and \cite[Prop.~11]{HeRiWa14} for the equivalence of the
approaches from \cite{HeRiWa14} and from the present paper in the 
particular case mentioned above.
Moreover, the components of the anchored decomposition of
$f$ are given by
\[
f_{\bspitch,\setu} = \int_{D^\setu} 
f^{(\setu)} (\bst_\setu;\bszero_{\setu^c}) \, K_{\bspitch,\setu}
(\cdot_\setu; \bst_{\setu}) \rd \bst_{\setu}
\] 
and the components of the ANOVA decomposition of $f$ are given
by
\[
f_{\an,\setu} = \int_{D^\setu} \! \int_{D^{\setu^c}}
f^{(\setu)} (\bst_\setu;\bst_{\setu^c})\rd \bst_{\setu^c}
\, K_{\an,\setu}
(\cdot_\setu; \bst_{\setu}) \rd \bst_{\setu},
\]
see \cite[Prop.~7]{HeRiWa14}.
These results should extend to the case of a general
domain $D$, a general weight function $\psi$, and arbitrary $p$ and $q$. 
\end{remark}

\section{Equivalence of Norms}\label{s4}

Let $p,q\in[1,\infty]$ and let 
$\bsgamma=(\gamma_\setu)_{\setu\subseteq[1:d]}$ be a family
of weights such that $\setU_\bsgamma \neq \emptyset$. Furthermore, 
we assume \eqref{eq2}.

Similar to \cite[Prop.~13]{HeRiWa14}, we have the following
proposition, whose proof is provided for completeness.

\begin{proposition}
The $\bsgamma$-weighted anchored and ANOVA spaces are equal, i.e.,
$W_{\bspitch,p,q,\psi,\bsgamma}=W_{\an,p,q,\psi,\bsgamma}$
(as vector spaces), if and only if the following holds:
\begin{equation}\label{monot}
  \gamma_\setw\,>\,0\quad\mbox{implies that}\quad
  \gamma_\setu\,>\,0\mbox{\ for all\ }\setu\subseteq\setw.
\end{equation}
Moreover, if \eqref{monot} does not hold then 
\[
  W_{\bspitch,p,q,\psi,\bsgamma}\,\not\subseteq\,  W_{A,p,q,\psi,\bsgamma}
  \quad\mbox{and}\quad
  W_{A,p,q,\psi,\bsgamma},\not\subseteq\,  W_{\bspitch,p,q,\psi,\bsgamma}. 
\]
\end{proposition}

\begin{proof}
As follows from 
Lemma \ref{lem:2}, every term $f_{\bspitch,\setw}$ 
in the anchored decomposition of $f \in W_{\bspitch,p,q,\psi,\bsgamma}$
is a linear combination of functions $f_{A,\setu} \in F_{\an,p,\psi,\setu}$
with $\setu\subseteq\setw$. 
Hence \eqref{monot} implies 
$W_{\bspitch,p,q,\psi,\bsgamma}\subseteq W_{A,p,q,\psi,\bsgamma}$. 
Using the same argument, we derive
$W_{A,p,q,\psi,\bsgamma}\subseteq W_{\bspitch,p,q,\psi,\bsgamma}$
from \eqref{monot}.

Suppose now that \eqref{monot} does not hold for some $\setw$. 
Consider $f(\bsx)=\prod_{j\in\setw}x_j$, 
which corresponds to $f = f_\bspitch$ with $g=1$ as well as
$\eta_\setw = 1$ and $\eta_\setu = 0$ for $\setu \neq \setw$ 
in Example \ref{ex4}. 
Therefore
$f \in W_{\bspitch,p,q,\psi,\bsgamma}$ with 
ANOVA decomposition 
\[
  f\,=\,\sum_{\setu\subseteq\setw} c^{|\setw|-|\setu|} \cdot
T_{A,\setu}(g_\setu).
\]
Since $c = m_\psi \neq 0$, see \eqref{mean}, and 
$T_{\an,\setu}(g_\setu) \neq 0$ for every $\setu$, but 
$\gamma_\setu=0$ for some $\setu \subseteq
\setw$, we obtain 
$f \not\in W_{A,p,q,\psi,\bsgamma}$.

The fact that $W_{A,p,q,\psi,\bsgamma}$ is not a subset of 
$W_{\bspitch,p,q,\psi,\bsgamma}$ can be shown in a similar way by 
considering $f(\bsx)=\prod_{j\in\setw}(x_j-m_\psi)$.
\end{proof}

From now on we assume that \eqref{monot} holds.
Let 
\[
  \imath_{p,q,\psi,\bsgamma}:     
W_{\an,p,q,\psi,\bsgamma}\hookrightarrow W_{\bspitch,p,q,\psi,\bsgamma}
	\qquad \text{and} \qquad 
	\imath_{p,q,\psi,\bsgamma}^{-1}:
W_{\bspitch,p,q,\psi,\bsgamma}\hookrightarrow W_{\an,p,q,\psi,\bsgamma}
\]
denote the embedding operators, which are continuous due to the
closed graph theorem. 
To simplify the notation we will often write 
\[
   \imath_{p,q}\qquad\mbox{and}\qquad \imath^{-1}_{p,q}.
\]
We are interested in the norms of these operators.

Let $\gc \in \{\bspitch,\an\}$.
Recall that, by definition, the operator
\begin{align*}
T_{\gc,\setu}\colon L_{p,\psi}(D^\setu) \to F_{\gc,p,\psi,\setu}
\end{align*}
is an isometric isomorphism.
Define the weighted space
\begin{align*}
\ell_{q,\bsgamma}\left((L_{p,\psi}(D^\setu))_{\setu}\right)
=\bigoplus_{\setu\in\setU_\bsgamma} L_{p,\psi}(D^\setu)
\end{align*}
endowed with the norm
\[
\left\|(g_\setu)_\setu\right\|
=\left|
\left( \|g_\setu\|_{L_{p,\psi}(D^\setu)} / \gamma_\setu \right)_\setu 
\right|_q.
\]
Again, by definition, the mapping
\begin{align*}
T_{\gc}\colon
\ell_{q,\bsgamma}\left((L_{p,\psi}(D^\setu))_\setu\right)
\to W_{\gc,p,q,\psi,\bsgamma},
\end{align*}
given by
\begin{align*}
T_{\gc}\left((g_\setu)_\setu\right)
=\sum_{\setu\in\setU_\bsgamma}T_{\gc,\setu}(g_\setu),
\end{align*}
is an isometric isomorphism.
Define the continuous operator
\begin{align*}
\jmath_{p,q}
=\jmath_{p,q,\psi,\bsgamma}
\colon
\ell_{q,\bsgamma}\left((L_{p,\psi}(D^\setu))_\setu \right)
\to
\ell_{q,\bsgamma}
\left((L_{p,\psi}(D^\setu))_\setu \right)
\end{align*}
by $\jmath_{p,q}=T_\bspitch^{-1}\circ \imath_{p,q}\circ T_\an$. In other 
words, $\jmath_{p,q}$ is defined such that the diagram
$$
\begin{tikzcd}[row sep=2.0em, column sep=2.0em]
  W_{\an,p,q,\psi,\bsgamma}      
\arrow{d}[swap]{\imath_{p,q}}   & 
\ell_{q,\bsgamma}
\left((L_{p,\psi}(D^\setu))_\setu\right)
\arrow{l}[swap]{T_\an} \arrow{d}{\jmath_{p,q}}     & \\
  W_{\bspitch,p,q,\psi,\bsgamma} 
\arrow{r}{T_\bspitch^{-1}}      & 
\ell_{q,\bsgamma}
\left((L_{p,\psi}(D^\setu))_\setu\right)
\end{tikzcd}
$$
is commutative.
Since $T_\bspitch$ and $T_\an$ are isometric isomorphisms, we have
\begin{align}\label{iequalj}
\|\jmath_{p,q}\|
=\|\imath_{p,q}\|
\qquad
\text{and}
\qquad
\|\jmath_{p,q}^{-1}\|
=\|\imath_{p,q}^{-1}\|.
\end{align}
Furthermore, \eqref{11} and \eqref{22} yield the explicit
representation
\begin{align}\label{33}
\jmath_{p,q}\left((g_{\setu})_\setu\right)
=\Bigl(
(-1)^{|\setu|}
\sum_{\setw\subseteq\setu^c}
\int_{D^\setw}(-1)^{|\setu\cup\setw|}g_{\setu\cup\setw}
(\cdot_\setu;\bst_\setw)\,
\overline{\psi}_\setw (\bst_\setw)\rd\bst_\setw
\Bigr)_\setu
\end{align}
and
\begin{align}\label{44}
\jmath_{p,q}^{-1}\left((g_{\setu})_\setu\right)
=\biggl(
\sum_{\setw\subseteq\setu^c}
\int_{D^\setw}g_{\setu\cup\setw}(\cdot_\setu;\bst_\setw)\,
\overline{\psi}_\setw (\bst_\setw)\rd\bst_\setw
\biggr)_\setu.
\end{align}

The following result was first obtained in \cite{KriPilWas16}
  for $D=[0,1]$ and $\psi=1$.

\begin{proposition}\label{prop:equalnorms}
 We have 
 \begin{align} 
   \| \imath_{p,q,\psi,\bsgamma} \|
   = \| \imath_{p,q,\psi,\bsgamma}^{-1} \|
   = \| \jmath_{p,q,\psi,\bsgamma} \|
   = \| \jmath_{p,q,\psi,\bsgamma}^{-1} \|.
 \end{align}
\end{proposition}

\begin{proof}
Define the operator
\begin{align*}
S\colon
\ell_{q,\bsgamma}\left((L_{p,\psi}(D^\setu))_\setu\right)
\to \ell_{q,\bsgamma}\left((L_{p,\psi}(D^\setu))_\setu\right)
\end{align*}
by
\begin{align*}
S\left((g_\setu)_\setu\right)
=\left((-1)^{|\setu|}g_\setu\right)_\setu.
\end{align*}
By definition of $S$ and by \eqref{33} as well as \eqref{44} the diagram
\[
\begin{tikzcd}[row sep=2.0em, column sep=2.0em]
  \ell_{q,\bsgamma}\left((L_{p,\psi}(D^\setu))_\setu\right) 
\arrow{d}[swap]{\jmath_{p,q}^{-1}} & 
  \ell_{q,\bsgamma}\left((L_{p,\psi}(D^\setu))_\setu\right) 
\arrow{l}[swap]{S} \arrow{d}{\jmath_{p,q}} & \\
  \ell_{q,\bsgamma}\left((L_{p,\psi}(D^\setu))_\setu\right) 
\arrow{r}{S}        
  &\ell_{q,\bsgamma}\left((L_{p,\psi}(D^\setu))_\setu\right) 
\end{tikzcd}
\]
is commutative.
Again $S$ is an isometric isomorphism. Hence we get the claim.
\end{proof}

Next we compute the norms of the embeddings in the extremal cases 
$p,q\in\{1,\infty\}$. Afterwards we use interpolation to find upper 
bounds for the general case.
In what follows we use the convention that $ \frac00 = 0$.

\subsection{The Case $p,q\in\{1,\infty\}$}
Recall that for $p=1$ we have
\[
\kappa_\psi \,= \,
\|\overline{\psi}/\psi \|_{L_\infty(D)}
\in (0, \infty),
\] 
see \eqref{eq2}. Furthermore, recall the definition of $m_\psi$
in \eqref{mean}.
In the particular case from Remark
\ref{r5}, where we have $m_\psi=1/2$ and $\kappa_\psi=1$,
the following result was obtained in
\cite[Thm.~14]{HeRiWa14}.

\begin{theorem}\label{thm:main}
For $p,q\in\{1,\infty\}$ we have
\[
  \|\imath_{p,q,\psi,\bsgamma}\| = 
  \|\imath_{p,q,\psi,\bsgamma}^{-1}\| =  C_{p,q,\psi,\bsgamma},
\]
where
\[
   C_{p,q,\psi,\bsgamma}\,=\,\left\{\begin{array}{ll}
 {\displaystyle \max_{\setu\subseteq[1:d]}\sum_{\setv\subseteq\setu^c}
    m_\psi^{|\setv|}\,\frac{\gamma_{\setu\cup\setv}}{\gamma_\setu}} &
   \mbox{for $p=\infty$ and $q=\infty$},\\
  \ \\
  {\displaystyle \max_{\setv\subseteq[1:d]}\sum_{\setu\subseteq\setv}
    m_\psi^{|\setv|-|\setu|}\,\frac{\gamma_{\setv}}{\gamma_\setu}} &
     \mbox{for\ $p=\infty$ and $q=1$},\\
  \ \\
{\displaystyle \max_{\setv\subseteq[1:d]} \sum_{\setu\subseteq\setv}
  \kappa_\psi^{|\setv|-|\setu|}\, \frac{\gamma_\setv}{\gamma_\setu}}
&\mbox{for $p=1$ and $q=1$},\\
   \ \\
  {\displaystyle \max_{\setu\subseteq[1:d]} 
   \sum_{\setv\subseteq\setu^c}\kappa_\psi^{|\setv|}\,
    \frac{\gamma_{\setu\cup\setv}}{\gamma_\setu}} &\mbox{for $p=1$ and 
     $q=\infty$}.
  \end{array}\right.
\]
\end{theorem}

\begin{proof}
According to Proposition \ref{prop:equalnorms} is suffices to
consider either $\imath_{p,q,\psi,\bsgamma}$ or 
$\imath_{p,q,\psi,\bsgamma}^{-1}$.
In the sequel,
\[
f = \sum_{\setu\in\setU_\bsgamma}T_{\an,\setu}(g_{\an,\setu})
= \sum_{\setu\in\setU_\bsgamma}T_{\bspitch,\setu}(g_{\bspitch,\setu})
\]
with $g_{\an,\setu}, g_{\bspitch,\setu} \in L_{p,\psi}(D^\setu)$.

\subsubsection*{\bf Case $p=q=\infty$}
Applying \eqref{22}, we get 
\begin{eqnarray*}
  \|f\|_{W_{A,\infty,\infty,\psi,\bsgamma}} &=&
\max_{\setu\in\setU_\bsgamma} 
\frac{\|g_{A,\setu}\|_{L_\infty(D^\setu)}}{\gamma_\setu}\\
&=& \max_{\setu\in\setU_\bsgamma}
   \esup_{\bsx_\setu\in D^\setu}\left|\frac1{\gamma_\setu}\,
    \sum_{\setv\subseteq\setu^c}\int_{D^\setv}g_{\bspitch,\setu\cup\setv}
    (\bsx_\setu;\bst_\setv)\,\overline{\psi}_\setv(\bst_\setv)
    \rd\bst_\setv\right|\\
   &\le& \max_{\setu \in \setU_\bsgamma} 
\max_{\setv\subseteq\setu^c}
   \frac{\|g_{\bspitch,\setu\cup\setv}\|_{L_{\infty(D^{\setu\cup\setv})}}}
  {\gamma_{\setu\cup\setv}}\cdot \left(\sum_{\setv\subseteq\setu^c}
    \frac{\bsgamma_{\setu\cup\setv}}{\gamma_\setu}\,
  m_\psi^{|\setv|} \right)\\
  &\le& \max_{\setw\in\setU_\bsgamma}
  \frac{\|g_{\bspitch,\setw}\|_{L_{\infty(D^\setw)}}}{\gamma_\setw}
  \cdot\max_{\setu\in\setU_\bsgamma}
 \sum_{\setv\subseteq\setu^c}m_\psi^{|\setv|}\,
    \frac{\gamma_{\setu\cup\setv}}{\gamma_\setu}\\
  &=& \|f\|_{W_{\bspitch,\infty,\infty,\psi,\bsgamma}}
\,C_{\infty,\infty,\psi,\bsgamma}.
\end{eqnarray*}
This proves that $\|\imath^{-1}_{\infty,\infty}\|\le
 C_{\infty,\infty,\psi,\bsgamma}$.

We next prove that there is equality. 
In fact, for $f_\bspitch$ according to Example \ref{ex4} 
with $g=1$ and $\eta_\setu = \gamma_\setu$ we obtain
$\|f_\bspitch\|_{W_{\bspitch,\infty,\infty,\psi,\bsgamma}}=1$
and $\|f_\bspitch\|_{W_{A,\infty,\infty,\psi,\bsgamma}}=
C_{\infty,\infty,\psi,\bsgamma}$.

\subsubsection*{\bf Case $p=\infty$ and $q=1$}
Applying \eqref{22} again, we have 
\begin{eqnarray*}
\|f\|_{W_{A,\infty,1,\psi,\bsgamma}}
   &=&\sum_{\setu\in\setU_\bsgamma} \frac1{\gamma_\setu}\,
    \|g_{A,\setu}\|_{L_{\infty}(D^{\setu})}\\
  &=& \sum_{\setu\in\setU_\bsgamma}\frac1{\gamma_\setu} \esup_{\bsx_\setu\in D^\setu}
    \left|\sum_{\setw\subseteq\setu^c}\int_{D^\setw}
      g_{\bspitch,\setu\cup\setw}(\bsx_\setu;\bst_\setw)\,
   \overline{\psi}_\setw(\bst_\setw)\rd\bst_\setw\right|\\
  &\le& \sum_{\setu\in\setU_\bsgamma}\frac1{\gamma_\setu}\,
  \sum_{\setw\subseteq\setu^c}\|g_{\bspitch,\setu\cup\setw}
  \|_{L_{\infty}(D^{\setu\cup\setw})}\,m_\psi^{|\setw|}
    \\
  &=&\sum_{\setu\in\setU_\bsgamma}\sum_{\setw\subseteq\setu^c}
   \frac{\|g_{\bspitch,\setu\cup\setw}\|_{L_\infty(D^{\setu\cup\setw})}}
      {\gamma_{\setu\cup\setw}}\,\frac{\gamma_{\setu\cup\setw}}{\gamma_\setu}
    \,m_\psi^{|\setw|}\\
  &\le& \|f\|_{W_{\bspitch,\infty,1,\psi,\bsgamma}}\,
C_{\infty,1,\psi,\bsgamma},
\end{eqnarray*}
which proves 
$\|\imath^{-1}_{\infty,1}\|\le C_{\infty,1,\psi,\bsgamma}$.

Let $\setw$ be such that 
\[
    \sum_{\setu\subseteq\setw}m_\psi^{|\setw|-|\setu|}\,
   \frac{\gamma_{\setw}}{\gamma_\setu}\,=\,C_{\infty,1,\psi,\bsgamma},
\]
and consider 
$f_\an$ according to Example \ref{ex4} 
with $g=1$ as well as  
$\eta_\setw = \gamma_\setw$ and $\eta_\setu = 0$ for
$\setu \neq \setw$.
Clearly, the ANOVA norm of $f_\an$ is equal to one, and
$\|f_\an\|_{W_{\bspitch,\infty,1,\psi,\bsgamma}} =
C_{\infty,1,\psi,\bsgamma}$.
Therefore $\|\imath_{\infty,1}\|\ge C_{\infty,1,\psi,\bsgamma}$.

\subsubsection*{\bf Case $p=q=1$}
Applying \eqref{22}, we have 
\begin{eqnarray*}
\|f\|_{W_{A,1,1,\psi,\bsgamma}}&=&
\sum_{\setu\in\setU_\bsgamma}\frac1{\gamma_\setu}
   \,\|g_{A,\setu}\|_{L_{1,\psi(D^\setu)}}\\
&=& \sum_{\setu\in\setU_\bsgamma}\frac1{\gamma_\setu}
   \int_{D^\setu}\left|\sum_{\setw\subseteq\setu^c}\int_{D^\setw}
  g_{\bspitch,\setu\cup\setw}(\bsx_\setu;\bst_\setw)\,
    \overline{\psi}_\setw(\bst_\setw)\rd\bst_\setw\right|
\psi_\setu(\bsx_\setu)
   \rd\bsx_\setu\\
 &\le&\sum_{\setu\in\setU_\bsgamma}\sum_{\setw\subseteq\setu^c}\int_{D^\setu}
   \int_{D^\setw}
  \frac{|g_{\bspitch,\setu\cup\setw}(\bsx_\setu;\bst_\setw)|}
   {\bsgamma_{\setu}}\, \psi_{\setu \cup \setw}
(\bsx_\setu;\bst_\setw) \, \frac{\overline{\psi}_{\setw}(\bst_\setw)}
{\psi_\setw(\bst_\setw)} \rd\bst_\setw \rd\bsx_\setu.
\end{eqnarray*}
Estimating 
$\overline{\psi}/\psi$ by $\kappa_\psi$ and replacing 
$\setu\cup\setw$ by $\setv$, we get that $\setw=\setv\setminus\setu$ 
and 
\begin{eqnarray*}
\|f\|_{W_{A,1,1,\psi,\bsgamma}}&\le&\sum_{\setv\in\setU_\bsgamma}
   \frac1{\gamma_\setv}\,\|g_{\bspitch,\setv}\|_{L_{1,\psi}(D^{\setv})}
   \cdot\left(\sum_{\setu\subseteq\setv}\kappa_\psi^{|\setv|-|\setu|}\,
    \frac{\gamma_\setv}{\gamma_\setu}\right)\\
   &\le& \|f\|_{W_{\bspitch,1,1,\psi,\bsgamma}}\,
C_{1,1,\psi,\bsgamma}.
\end{eqnarray*}
This proves that $\|\imath_{1,1}^{-1}\|\le C_{1,1,\psi,\bsgamma}$.

We now show that $\|\imath^{-1}_{1,1}\|= C_{1,1,\psi,\bsgamma}$. 
Consider a sequence of non-negative
functions $G_n \in L_1(D)$ such that 
\[
  \int_D G_n(x)\rd x\,=\,1\quad\mbox{and}\quad
  \lim_{n\to\infty}\int_D G_n(x)\,\frac{\overline{\psi}(x)}{\psi(x)}
    \rd x\,=\,\kappa_\psi.  
\]
For instance 
\begin{equation}\label{Gn}
   G_n(x)\,=\,\frac1{\lambda(K_n)}\,{\bf 1}_{K_n}(x),\quad
   \mbox{where}\quad 
   K_n\,\subseteq\,\left\{x\in D\ :\ 
\frac{\overline{\psi}(x)}{\psi(x)}\ge 
   \kappa_\psi-\frac1n\right\},
\end{equation}
has  positive and finite Lebesgue measure $\lambda(K_n)$.

Define
\begin{equation}\label{gn}
  g_n(t)\,=\,\frac{G_n(t)}{\psi(t)}.
\end{equation}
Of course, $\|g_n\|_{L_{1,\psi}(D)} = 1$.
Moreover, $m_n$ defined by
\[
  m_n
\, =\,  \int_Dg_n(t)\,\overline{\psi}(t) \rd t
\]
satisfies $\lim_{n\to\infty} m_n\,=\,\kappa_\psi$.

Let $\setw\subseteq[1:d]$ be such that 
\[
  C_{1,1,\psi,\bsgamma}\,=\,
\sum_{\setu \subseteq\setw}\kappa_\psi^{|\setw|-|\setu|}\,
   \frac{\gamma_\setw}{\gamma_\setu}.
\]
Let $f_\bspitch$ be given according to Example \ref{ex4} with
$g=g_n$, $\eta_\setw = \gamma_\setw$, and $\eta_\setu = 0$ for
$\setu \neq \setw$.
It follows that $f_\bspitch \in W_{\bspitch,1,1,\psi,\bsgamma}$ with 
$\|f_\bspitch\|_{W_{\bspitch,1,1,\psi,\bsgamma}} =1$ and
$c = m_n$.
Moreover, the ANOVA decomposition of $f_\bspitch$ is given by
\[
f_\bspitch \,=\,
\gamma_\setw \,
\sum_{\setu \subseteq \setw} m_n^{|\setw|-|\setu|} \,
T_{\an,\setu} (g_{\setu}).
\]
Consequently,
\[
  \|f_\bspitch\|_{W_{A,1,1,\psi,\bsgamma}}\,=\,
\sum_{\setu\subseteq\setw}\,
  m_n^{|\setw|-|\setu|}\,\frac{\gamma_\setw}{\gamma_\setu},
\]
which converges to $C_{1,1,\psi,\bsgamma}$ as $n \to \infty$.
This completes the proof that $\|\imath^{-1}_{1,1}\|
= C_{1,1,\psi,\bsgamma}$.

\subsubsection*{\bf Case $p=1$ and $q=\infty$}
Here we have
\begin{eqnarray*}
 &&\|f\|_{W_{A,1,\infty,\psi,\bsgamma}}\\
 &&=\, \max_{\setu\in\setU_\bsgamma}\frac1{\gamma_\setu}
   \int_{D^\setu}\left|\sum_{\setv\subseteq\setu^c}\int_{D^\setv}
   g_{\pitchfork,\setu\cup\setv}(\bsx_\setu;\bst_\setv)
\, \psi_\setv(\bst_\setv)
     \cdot
    \frac{\overline{\psi}_\setv(\bst_\setv)}
{\psi_\setv(\bst_\setv)}\rd\bst_\setv
    \right| \psi_\setu (\bsx_\setu) \rd\bsx_\setu\\
  &&\le\,  \max_{\setu\subseteq[1:d]}\sum_{\setv\subseteq\setu^c}
    \frac{\|g_{\bspitch,\setu\cup\setv}\|_{L_{1,\psi}
(D^{\setu\cup\setv})}}
     {\gamma_{\setu\cup\setv}}\,\frac{\gamma_{\setu\cup\setv}}{\gamma_\setu}\,
      \kappa_\psi^{|\setv|}\\
  &&\le\, C_{1,\infty,\psi,\bsgamma}\,\|f\|_{W_{\bspitch,1,\infty,\psi,\bsgamma}}
\end{eqnarray*}
and $\|\imath^{-1}_{1,\infty}\|\le C_{1,\infty,\psi,\bsgamma}$,
as needed. 

To prove that 
$\|\imath_{1,\infty}\|$ is 
equal to 
$C_{1,\infty,\psi,\bsgamma}$ it is enough to consider 
$f_A$ as in Example \ref{ex4} with $g=g_n$ as in (\ref{gn})
and $\eta_\setu = (-1)^{|\setu|} \gamma_\setu$. Then $\|f_A\|_{W_{A,1,\infty,\psi,\bsgamma}} = 1$.
By using the anchored decomposition of $f_A$, we see that 
\begin{equation}
\|f_A\|_{W_{\bspitch,1,\infty,\psi,\bsgamma}} = \max_{\setu \subseteq [1:d]} \sum_{\setv \subseteq \setu^c}
m_n^{|\setv|} \frac{\gamma_{\setu \cup \setv}}{\gamma_{\setu}},
\end{equation}
which converges to $C_{1,\infty, \psi, \bsgamma}$ as $n\to\infty$.
\end{proof}

\subsection{General Case $p,q\in[1,\infty]$}
In this section we follow the approach from \cite{HS16}, as 
we use complex interpolation to obtain upper bounds 
for
$\| \imath_{p,q,\psi,\bsgamma} \|= \| \imath_{p,q,\psi,\bsgamma}^{-1} \|$
for $p,q\in[1,\infty]$. We assume here that \eqref{eq2} holds for $p=1$.
Then it also holds for $p>1$, so all function spaces considered 
below are well defined.

We work with complex valued functions since the direct application of
the complex interpolation method needs complex scalars. By considering
real and imaginary parts separately, the definition of the spaces
$W_{\gc,p,q,\psi,\bsgamma}$ can be extended to complex valued
functions on $D^d$.
Derivatives and integrals are applied to
both parts. The results of the previous subsection remain valid.
Indeed, the lower bounds for the norms obviously also hold for
complex valued functions. Moreover, the proofs of the upper bounds
remain valid also in the complex case. This is due to the fact that 
the inequalities used are triangle inequalities,
which are
also valid for complex scalars.

By Proposition \ref{prop:equalnorms} it is enough to consider
$\| \imath_{p,q,\psi,\bsgamma}\|$. The next theorem
provides the general interpolation result for these norms,
cf.\ \cite[Sec.~4]{HS16}.

\begin{theorem}\label{thm:generalinterpol}
Let $p,q,p_0,q_0,p_1,q_1 \in [1,\infty]$ and 
$\theta\in[0,1]$ be such that 
\[
\frac{1}{p}=\frac{1-\theta}{p_0}+\frac{\theta}{p_1} \quad  \text{and}
\quad \frac{1}{q}=\frac{1-\theta}{q_0}+\frac{\theta}{q_1}.
\]
Then 
\[
  \| \imath_{p,q,\psi,\bsgamma}\| \le 
  \| \imath_{p_0,q_0,\psi,\bsgamma}\|^{1-\theta} \,
  \| \imath_{p_1,q_1,\psi,\bsgamma}\|^\theta.
\]
\end{theorem}

\begin{proof}
By Proposition \ref{prop:equalnorms}, it is enough to prove the
corresponding result for $\jmath_{p,q,\psi,\bsgamma}$ instead of
$\imath_{p,q,\psi,\bsgamma}$. The relevant results for spaces of type
$\ell_{q,\bsgamma} \left((A_\setu)_\setu\right)$ 
with $A_\setu=L_{p,\psi}(D^\setu)$
are Theorem 1.18.1 (formula (4)) in \cite{Tr78}, i.e.,
\begin{equation}\label{intpol1}
\left[\ell_{q_0}(A_j),\ell_{q_1}(B_j)\right]_\theta\,=\,
\ell_q\left([A_j,B_j]_\theta\right),
\end{equation}
for Banach spaces $A_j,B_j$ and $1/q=(1-\theta)/q_0+\theta/q_1$,
complemented by the following Remark 2, and Theorem 1.18.6.2
(formula (15)) in \cite{Tr78}, i.e.,
\begin{equation}\label{intpol2}
[L_{p_0}(A),L_{p_1}(A)]_\theta=L_p(A),
\end{equation}
for  a Banach space $A$ and $1/p=(1-\theta)/p_0+\theta/p_1$ with
$0<\theta<1$. All these interpolation identities are to be
understood with equality of the norms. Together they prove the
claim in the theorem.
\end{proof}

\begin{remark}
For 
future
applications it might be useful to use different
weight sequences for different $p,q$. As explained in \cite{HS16},
this is possible without much difficulties by also interpolating
the weights. Additionally, it is also
possible to interpolate the weight functions $\psi$.
\end{remark}

In the particular case where $D=[0,1]$, $\psi=1$ and $p=q$ 
the upper bound for 
$\|\imath_{p,p,\psi,\bsgamma}\|$ 
from the next result was obtained in \cite[Thm.~2]{HS16}.

\begin{theorem}\label{thm:main2}
For $1 \leq p \leq q \leq \infty$ we have
\[
  \|\imath_{p,q,\psi,\bsgamma}\| \,= \,
  \|\imath_{p,q,\psi,\bsgamma}^{-1}\| \, \leq\,
   C_{1,\infty,\psi,\bsgamma}^{1/p-1/q} \cdot C_{1,1,\psi,\bsgamma}^{1/q}
   \cdot C_{\infty,\infty,\psi,\bsgamma}^{1-1/p}. 
\]
For $1 \leq q \leq p \leq \infty$ we have
\[
  \|\imath_{p,q,\psi,\bsgamma}\| \,= \,
  \|\imath_{p,q,\psi,\bsgamma}^{-1}\| \, \leq\,
   C_{\infty,1,\psi,\bsgamma}^{1/q-1/p} \cdot C_{1,1,\psi,\bsgamma}^{1/p}
   \cdot C_{\infty,\infty,\psi,\bsgamma}^{1-1/q}.
\]
\end{theorem}

\begin{proof}
For simplicity, we abbreviate $C_{p,q}=C_{p,q,\psi,\bsgamma}$. In a
first step, applying Theorem \ref{thm:generalinterpol} in the case
$p_0=q_0=\infty, p_1=q_1=1, \theta=\frac1p$, we get 
\[ 
  \|\imath_{p,p,\psi,\bsgamma}\|
 \,\le\, C_{1,1}^{1/p} \cdot C_{\infty,\infty}^{1-1/p}.
\]
In the case $p<q$, determine $r\in [1,\infty]$ and $\theta\in[0,1]$ via
the equations
\[
\frac{1}{p}=\frac{1-\theta}{1}+\frac{\theta}{r} \quad  \text{and}
\quad \frac{1}{q}=\frac{1-\theta}{\infty}+\frac{\theta}{r}
\]
and obtain, again applying Theorem \ref{thm:generalinterpol},
\[ 
  \|\imath_{p,q,\psi,\bsgamma}\|
 \,\le\, C_{1,\infty}^{1-\theta} \cdot C_{r,r}^{\theta}
 \,\le\, C_{1,\infty}^{1-\theta} \cdot
C_{1,1}^{\theta/r} \cdot C_{\infty,\infty}^{\theta(1-1/r)} \,=\,
 C_{1,\infty}^{1/p-1/q} \cdot C_{1,1}^{1/q} \cdot 
C_{\infty,\infty}^{1-1/p}. 
\]
In the case $q<p$, we similarly get
\[ 
  \|\imath_{p,q,\psi,\bsgamma}\|
 \, \le \,
 C_{\infty,1}^{1/q-1/p} \cdot C_{1,1}^{1/p} \cdot 
C_{\infty,\infty}^{1-1/q}. 
\]
\end{proof}

Adopting the proof technique of a part of Theorem 1 in 
\cite{KriPilWas16}, we can show the following lower bound.
For notational convenience we put $\psi^{1/\infty}=1$.

\begin{theorem}\label{thm:low}
Let
\[
B_p\,=\,
\|\overline{\psi}/\psi^{1/p}\|_{L_{p^\prime}(D)}.
\]
For all $p,q\in[1,\infty]$
\[
\|\imath_{p,q,\psi,\bsgamma}\| \,\ge\,
\sup
\frac{
\left|
\left(
\sum_{\setw\subseteq\setu^c}c_{\setu\cup\setw}\,
\gamma_{\setu\cup\setw}\,B_p^{|\setw|} / \gamma_\setu
\right)_\setu
\right|_q}
{\left| \left(c_\setu\right)_\setu \right|_q},
\]
where the supremum is taken over all families $(c_\setu)_\setu =
(c_u)_{\setu\in\setU_\bsgamma}$ of non-negative real numbers.
\end{theorem}

\begin{proof} 
Consider a sequence of functions $G_n \in L_p(D)$ such that
$\|G_n\|_{L_{p}(D)}=1$ and 
\[
\lim_{n\to\infty}
\int_D G_n(t)\,
\frac{\overline{\psi}(t)}{\psi^{1/p}(t)} \rd t\,=\, B_p.
\]
In the case $p=1$ we may choose this sequence as in (\ref{Gn}).
In the case $p>1$ it 
suffices to consider a single function $G=G_n$, since
$L_p(D)$ is the dual space of $L_{p^\prime}(D)$.
Define
\[
g_n(t) \,=\, \frac{G_n(t)}{\psi^{1/p}(t)}
\]
and
\[
m_n\,=\,\int_D g_n(t) \, \overline{\psi}(t)\rd t 
\]
to obtain $g_n \in L_{p,\psi}(D)$ with $\|g_n\|_{L_{p,\psi}(D)}=1$
and $\lim_{n\to\infty} m_n = B_p$.
Let $f_\bspitch$ be given by Example \ref{ex4} with $g=g_n$ and
$\eta_\setu=c_\setu \, \gamma_\setu$, where $c_\setu \in \bbR$.
Clearly
$c=m_n$ and
\[
\|f_\bspitch\|_{W_{\bspitch,p,q,\psi,\bsgamma}}\,=\,
\left| \left(c_\setu\right)_\setu \right|_q.
\]
Furthermore,
\[
\|f_\bspitch\|_{W_{A,p,q,\psi,\bsgamma}}\,=\,
\left| \left(
\sum_{\setw\subseteq\setu^c} c_{\setu\cup\setw}\,
\gamma_{\setu\cup\setw} \,m_n^{|\setw|} / \gamma_\setu \right)_\setu 
\right|_q.
\]
Let $n$ tend to $\infty$ to complete the proof.
\end{proof}

\section{Uniform and Polynomial Equivalence of Norms}\label{s5}

So far the number $d$ of variables was fixed. 
In this section we 
study the equivalence of the norm for varying $d$.
More precisely, for $d\in\bbN$ we have a family 
\[
  \bsgamma^{[d]}\,=\,(\gamma_{d,\setu})_{\setu\subseteq[1:d]}
\]
of weights that satisfy \eqref{monot} for every $d$.
Furthermore, we assume that \eqref{eq2} is satisfied
so that 
\[
   W_{\bspitch,p,q,\psi,\bsgamma^{[d]}}
   \,=\,W_{A,p,q,\psi,\bsgamma^{[d]}}\quad
   \mbox{for all}\ d\in\bbN
\]
as vector spaces. Of course, the norm
\[
  \|\imath_{p,q,\psi,\bsgamma^{[d]}}\| \,=\,
  \|\imath_{p,q,\psi,\bsgamma^{[d]}}^{-1}\|
\]
of the embeddings depends on $d$ in all non-trivial cases. 

\begin{definition}\hfill
\begin{description}
\item[{(i)}] 
The weighted anchored and ANOVA norms are
\emph{uniformly equivalent} if
\[
\sup_{d\in\bbN} \|\imath_{p,q,\psi,\bsgamma^{[d]}}\|
\,<\,\infty.
\]
\item[{(ii)}] 
The weighted anchored and ANOVA norms are \emph{polynomially  
equivalent} if there exist $\tau>0$ such that 
\[
\|\imath_{p,q,\psi,\bsgamma^{[d]}}\|
\,=\,
O\left(d^{\,\tau}\right). 
\]
The smallest (or infimum of) such $\tau$ is called the exponent of 
polynomial equivalence of the norms. 
\end{description}
\end{definition}

We now consider special classes of weights to see when there is
uniform or polynomial equivalence. 
The 
explicit formulas for 
$C_{p,q,\psi,\bsgamma^{[d]}}
=\|\imath_{p,q,\psi,\bsgamma^{[d]}}\|$ 
for $p,q\in\{1,\infty\}$ according to Theorem \ref{thm:main}
are very similar to those obtained in \cite{HeRiWa14} for $D=[0,1]$
and $\psi = 1$. Since applying them with Theorem
\ref{thm:main2}
is rather straightforward, we will omit the proofs of upper bounds of
$\|\imath_{p,q,\psi,\bsgamma^{[d]}}\|$.
Corresponding lower bounds can be
obtained by applying techniques from \cite{KriPilWas16} to
Theorem \ref{thm:low}. This is why proofs of some lower bounds are 
also omitted. We begin with the product weights that are the most
commonly used in the literature. 

\vskip 1pc
\emph{Product weights}, introduced in \cite{SloWoz}, 
have the following form
\[
  \gamma_\setu\,=\,\prod_{j\in\setu}\gamma_j
  \quad\mbox{for positive numbers\ }\gamma_j.
\]
We have 
\[
C_{\infty,1,\psi,\bsgamma^{[d]}}
\,=\,C_{\infty,\infty,\psi,\bsgamma^{[d]}}
\,=\,\prod_{j=1}^d(1+m_\psi\,\gamma_j)
\]
and
\[
C_{1,1,\psi,\bsgamma^{[d]}}\,=\,C_{1,\infty,\psi,\bsgamma^{[d]}}
\,=\,\prod_{j=1}^d(1+\kappa_\psi\,\gamma_j). 
\]
Hence the conditions
\begin{equation}\label{prod1}
\sum_{j=1}^\infty\gamma_j\,<\,\infty\quad\mbox{and}\quad
\sup_{d\in\bbN}\frac{\sum_{j=1}^{d}\gamma_j}{\ln(d+1)}\,<\,\infty
\end{equation}
are sufficient for the corresponding uniform and polynomial equivalences
for all $p,q \in [1,\infty]$.
From Theorem \ref{thm:low}, one can conclude that \eqref{prod1} are
also necessary for corresponding uniform and polynomial
equivalences
for any $p,q \in [1,\infty]$, cf.~\cite[Prop.~3]{KriPilWas16}. 

\vskip 1pc
\emph{Product order-dependent weights}, 
introduced in \cite{KSS}, have the form
\[
  \gamma_{d,\setu}\,=\,\left(|\setu|!\right)^{\beta_1}\cdot
   \prod_{j\in\setu}\frac{c}{j^{\beta_2}}\quad{with}\quad
   0\,<\,\beta_1\,<\,\beta_2\quad\mbox{and}\quad c>0. 
\]
The following lower bound holds for every $p,q$ and $\tau>0$
\[
\|\imath_{p,q,\psi,\bsgamma^{[d]}}\|
   \,=\,\Omega(d^{\,\tau}),
\]
    cf.\ \cite[Prop.~18]{HeRiWa14} for the extremal cases $p=q=1$
and $p=q=\infty$ and cf. \cite[Prop.~8]{KriPilWas16} for the 
general case.
Hence there is no polynomial equivalence for any $p,q \in [1,\infty]$. 
\vskip 1pc
\emph{Finite order weights}, introduced in \cite{DSWW}, 
are such that 
\[
   \gamma_{d,\setu}\,=\,0\quad\mbox{if}\quad|\setu|\,>\,r
\]
for a given fixed number $r\ge1$. We consider the following
special weights
\begin{equation}\label{eq:spfinord}
  \gamma_\setu\,=\,\omega^{|\setu|}\quad\mbox{if}\quad|\setu|\,\le\,r,
\end{equation}
where $\omega$ is a positive number. For $q=1$, we have 
\[
   C_{1,1,\psi,\bsgamma^{[d]}}\,=\,(1+\kappa_\psi\,\omega)^r 
  \qquad\mbox{and}\qquad
C_{\infty,1,\psi,\bsgamma^{[d]}}\,=\,(1+m_\psi\,\omega)^r,
\]
whereas for $q=\infty$, we have 
\[
C_{p,\infty,\psi,\bsgamma^{[d]}}\,=\,\Theta(d^r)  \quad\mbox{for\ } 
p\in\{1,\infty\}.
\]
Using Theorem \ref{thm:main2} we conclude that 
for arbitrary $p$ and
$q$ 
\[
\|\imath_{p,q,\psi,\bsgamma^{[d]}}\|
\,=\,O\left(d^{r(1-1/q)}\right).
\]
Hence we have at least polynomial equivalence with the exponent
bounded by $r(1-1/q)$. For $q=1$, we have uniform equivalence. 
Actually, the exponent of polynomial equivalence is precisely $r(1-1/q)$. 
Indeed, use the lower bound in Theorem \ref{thm:low} for $c_\setu=1$ 
if $|\setu|=r$ and $c_\setu=0$ otherwise. 
If $|u| \leq r$ we obtain
\[
\sum_{\setw\subseteq\setu^c}c_{\setu\cup\setw}\,
\gamma_{\setu\cup\setw}\,B_p^{|\setw|} / \gamma_\setu
\,=\,
(\omega B_p)^{r-|\setu|} \cdot
|\{\setw\subseteq\setu^c : |\setw|=r-|\setu|\}|
\,=\, (\omega B_p)^{r-|\setu|} \cdot \binom{d-|\setu|}{r-|\setu|}.
\]
For $1 \leq q < \infty$ this yields
\[
\|\imath_{p,q,\psi,\bsgamma^{[d]}}\| \,\geq\,
\frac{
\left(
\sum_{s=0}^r
(\omega B_p)^{q\,(r-s)} \cdot
\binom{d-s}{r-s}^q \cdot \binom{d}{s} \right)^{1/q}}
{\binom{d}{r}^{1/q}}.
\]
Considering only the term with $s=0$ this gives
\[
\|\imath_{p,q,\psi,\bsgamma^{[d]}}\| \,\geq\,
(\omega B_p)^r \cdot \binom{d}{r}^{1-1/q}
\,=\,\Omega\left(d^{r\,(1-1/q)}\right).
\]
This estimate is valid for $q=\infty$, too.

\vskip 1pc
Consider \emph{finite diameter weights}, introduced by Creutzig (see 
\cite{Cre07} and \cite{NowWoz08}), of the form
\[
  \gamma_\setu\,=\,\left\{\begin{array}{ll} \omega^{|\setu|} 
 &\mbox{if\ }{\rm diam}(\setu)\,\le\,r,\\
 0  &\mbox{if\ }{\rm diam}(\setu)\,>r, \end{array}\right.
\] 
where ${\rm diam}(\setu)=\max_{i,j\in\setu}(i-j)$, where 
${\rm diam}(\emptyset)=0$, by convention.
As in \cite{KriPilWas16}, 
\[
C_{1,1,\psi,\bsgamma^{[d]}}\,=\,(1+\omega\,\kappa_\psi)^{r+1}
\quad\mbox{and}\quad
C_{\infty,1,\psi,\bsgamma^{[d]}}\,=\,(1+\omega\,m_\psi)^{r+1},
\]
whereas
\[
  C_{p,\infty,\psi,\bsgamma^{[d]}}\,=\,\Theta(d)\quad\mbox{for\ }
p\in\{1,\infty\}.
\]
By applying interpolation we get 
$\|\imath_{p,q,\psi,\bsgamma^{[d]}}\|=O\left(d^{1-1/q}\right)$ 
for all $p,q$. Similar to the proof of Proposition 7 in 
\cite{KriPilWas16} one can show that the above bound is sharp, i.e., 
\[
\|\imath_{p,q,\psi,\bsgamma^{[d]}}\|
\,=\,\Theta\left(d^{1-1/q}\right),
\]
which means polynomial equivalence with the exponent $1-1/q$.

\vskip 1pc
Finally, consider special \emph{dimension-dependent weights} 
\begin{equation}\label{creazy}
   \gamma_{d,\setu}\,=\,d^{-|\setu|}.
\end{equation}
introduced in \cite{HegWas}. Then for $q\in\{1,\infty\}$, 
\[
   C_{1,q,\psi,\bsgamma^{[d]}}\,=\,\left(1+\kappa_\psi/d\right)^d\,\le\,
    \exp\left(\kappa_\psi\right)
   \quad\mbox{and}\quad
   C_{\infty,q,\psi,\bsgamma^{[d]}}\,=\,\left(1+m_\psi/{d}\right)^d\,\le\,
     \exp\left(m_\psi\right). 
\]
The interpolation yields uniform equivalence for all $p$ and $q$.

\section{Applications to Integration and Approximation}\label{s6}

A thorough study of applications of embedding results to high-
and infinite-dimensional integration in the
setting of reproducing kernel Hilbert spaces with product weights
is carried out in \cite{GnHeHiRi16}. This abstract approach
covers the particular case $p=q=2$ with product weights
$\gamma_\setu$ in the setting of the present paper. 
A number of new error estimates and new tractability results could be 
obtained in \cite{GnHeHiRi16}
by transferring known results for the anchored setting to the 
ANOVA setting or vice versa. Roughly speaking, the anchored setting
is known to be very well suited for the analysis of deterministic
algorithms, while the ANOVA setting is much preferable for the 
analysis of randomized algorithms.

The results of the present paper allow to transfer results 
between the anchored and the ANOVA setting beyond Hilbert
spaces and product weights.
Unfortunately, we are only aware of few results for the non-Hilbert space 
setting or the corresponding weighted discrepancies, see 
\cite{Ai14,HiPiSi08,KriPilWas16,KrPiWa16b,LeoPil03,SloWoz,Was14}.
In the sequel we discuss the transfer of results from
\cite{HiPiSi08,KriPilWas16,KrPiWa16b,Was14} from the anchored
setting to the ANOVA setting.

At first, we illustrate how to transfer
the tractability results from \cite{HiPiSi08}.
The results there are formulated in terms of
the weighted star discrepancy. Via Koksma-Hlawka duality, this 
corresponds to results for uniform integration on $D^d = [0,1]^d$ in 
$W_{\bspitch,1,1,\psi,\bsgamma}$ in the case $\psi=1$. 

For product weights satisfying the condition 
$\sum_{j=1}^\infty\gamma_j<\infty$, Theorem 3 in 
\cite{HiPiSi08} shows that this problem is strongly tractable, i.e.,
the number $N(\varepsilon,d)$ of sample points needed to achieve an error $
\varepsilon>0$ can be bounded by $C \varepsilon^{-\beta}$ with absolute 
constants $C,\beta>0$ independent of the dimension $d$. 
Moreover, it is also 
shown there that the exponent of strong tractability, that is the 
infimum over all possible $\beta$, is 1. 
The results in Section \ref{s5} show that 
anchored and ANOVA norms are uniformly equivalent. Hence we immediately
obtain that we also have strong tractability of uniform integration on 
$[0,1]^d$ in  $W_{A,1,1,\psi,\bsgamma}$ in the case $\psi=1$. The 
algorithms achieving this are QMC-algorithms using superpositions of 
digital nets over 
$\mathbb{Z}_2$. For details of the construction we refer to \cite{HiPiSi08}
and the references therein.

For general weights, Corollary 1 in \cite{HiPiSi08} shows that integration 
is polynomially tractable with $\varepsilon$-exponent 2 and $d$-exponent 0
under the condition
\[
  C_\gamma \,=\, \sup_{d\in\bbN} \max_{\setu\subseteq[1:d]} 
\gamma_{d,\setu} \sqrt{|\setu|} < \infty. 
\]
This condition is satisfied for bounded finite order weights, for finite 
diameter weights and for the dimension dependent weights in \eqref{creazy}.
More exactly, we have the estimate
\[
 N(\varepsilon,d) \,\le C\, C_\gamma \, (1+\log d) \,\varepsilon^{-2} 
\]
with some constant $C>0$. 
From the results in Section \ref{s5} we infer that the same holds for
uniform integration on $[0,1]^d$ in  $W_{A,1,1,\psi,\bsgamma}$ in the 
case $\psi=1$ for the special finite order weights defined in 
\eqref{eq:spfinord}, for finite diameter weights, and for the special 
dimension dependent weights \eqref{creazy}.

The anchored spaces studied in this paper were also considered in
\cite{Was14} in the context of function approximation. Actually,
functions of infinitely-many variables are studied in \cite{Was14},
however, if we define $\gamma_\setu = 0$ for
every finite set $\setu \subseteq \bbN$ with $\setu \setminus [1:d]
\neq \emptyset$ in the setting of the latter paper, we obtain
functions on $D^d$. For a given probability density $\omega:D\to\bbR_+$ 
and a real $1 \leq s \leq \infty$ one
is interested in approximating $f\in W_{\bspitch,p,q,\psi,\bsgamma}$
with the error measured in a norm $\|\cdot\|_G$
satisfying the following condition. If $s < \infty$, then
\[
  \|f \|_G\,\le\,
  \left(\sum_{\setu\in\setU_\bsgamma}\|f_{\bspitch,\setu}\|_{G_\setu}^s
\right)^{1/s}
\]
for the anchored decomposition 
\[
f = \sum_{\setu\in\setU_\bsgamma}f_{\bspitch,\setu},
\]
where
\[
  \|f_\setu\|_{G_\setu}=\left(\int_{D^{|\setu|}}|f_\setu(\bsx_\setu)|^s
  \,\prod_{j\in\setu}\omega(x_j)\rd\bsx_\setu\right)^{1/s}.
\]
For $s=\infty$ the condition is modified in the usual way.
Almost optimal algorithms and sharp complexity bounds for such 
approximation problems were derived in \cite{Was14}; however, only for 
the anchored spaces. The uniform equivalence studied in the current paper 
allows to transfer the result of \cite{Was14} to the case of ANOVA spaces
resulting in almost optimal algorithms and sharp complexity bounds.
The multivariate decomposition method, which is particularly tuned to the 
anchored setting at a first glance, turns out to be almost
optimal also in the ANOVA setting. 

Now we turn to the results from \cite{KriPilWas16,KrPiWa16b}. 
For problems with large number $d$ of variables, one may try to replace
the original functions $f$ by functions $f_k$ with only $k \ll d$ 
variables, namely,
\[f_k(\bsx)\,=\,f(x_1,\dots,x_k,0,0,\dots,0).
\]
As shown in \cite{KriPilWas16,KrPiWa16b}, for weighted integration 
and weighted $L_s$ approximation, as above, and for modest error
demands $\e$, one can truncate the dimension with $k=k(\e)$ being
very small. This holds for problems defined on anchored spaces
$W_{\bspitch,p,q,\psi,\bsgamma}$. Moreover, in general, this desirable
property does not hold for ANOVA spaces $W_{A,p,q,\psi,\bsgamma}$.
However, if the anchored and ANOVA spaces are uniformly equivalent then
also in the setting of ANOVA spaces, one can use functions $f_{k(\e)}$
with $k(\e)$ only slightly larger than the corresponding truncation
dimension for the anchored spaces.

\section{Appendix}

As previously, $1\le p\le \infty$, and $p^\prime$ denotes its
conjugate, $1/p+1/p^\prime=1$. Furthermore, let $(\Omega,\mathfrak{A},\mu)$
denote a probability space.

\begin{lemma}\label{prop1}
Let $g : \Omega \to \bbR$ be a measurable function.
Then we have $g\in L_{p^\prime}(\mu)$ if and only if
$fg \in L_1(\mu)$ for all $f\in L_p(\mu)$.
\end{lemma}

\begin{proof}
Suppose that
$fg \in L_1(\mu)$ for all $f\in L_p(\mu)$.
For $p=\infty$ the function $f$, given by 
$f(x) = g(x)/|g(x)|$ if $g(x) \neq 0$ and $f(x)=0$ otherwise, is in 
$L_\infty(\mu)$. Hence $|g|= fg \in L_1(\mu)$, implying 
$g\in L_1(\mu)$.

Now let $1\le p <\infty$.
For $n\in\bbN$ let $g_n:\Omega \to \bbR$ be the function equal to 
$g(x)$ if $|g(x)|\le n$ and $g_n(x)=0$ otherwise. Then $g_n$ is bounded 
and $g_n\to g$ almost everywhere. In particular 
$g_n\in L_{p^\prime}(\mu)$, 
so the functionals $I_n$ defined by $I_n f = \int_\Omega
f g_n \, {\rm d}\mu$
are linear and bounded on $L_p(\mu)$. 
By assumption and the dominated convergence theorem, the limit
\begin{align*}
\lim_{n\to\infty} I_nf
=\lim_{n\to\infty} \int_\Omega f g_n \, {\rm d} \mu
= \int_\Omega fg \, {\rm d}\mu
\end{align*}
exists and is finite for every $f\in L_p(\mu)$. So the bounded
linear functionals $I_n$ converge pointwise to the linear
functional $I$  given  by $I f = \int_\Omega f g \, {\rm d}\mu$. 
The Banach-Steinhaus Theorem implies that $I$ is a bounded functional 
on $L_p(\mu)$.
Since $L_{p'}(\mu)$ is the dual space of
$L_p(\mu)$, we obtain $g\in L_{p^\prime}(\mu)$. 

H\"older's inequality immediately yields the reverse implication.
\end{proof}

\newcommand\BIT{\emph{BIT\ }}
\newcommand\Com{\emph{Computing\ }}
\newcommand\CA{\emph{Constr. Approx.\ }}
\newcommand\FCM{\emph{Found. Comput. Math.\ }}
\newcommand\JAT{\emph{J. Approx. Th.\ }}
\newcommand\JC{\emph{J. Complexity\ }}
\newcommand\JCP{\emph{J. of Computational Physics\ }}
\newcommand\JMA{\emph{SIAM J. Math. Anal.\ }}
\newcommand\JMAA{\emph{J. Math. Anal. Appl.\ }}
\newcommand\JMM{\emph{J. Math. Mech.\ }}
\newcommand\JMP{\emph{J. Math. Physics\ }}
\newcommand\MC{\emph{Math. Comp.\ }}
\newcommand\NM{\emph{Numer. Math.\ }}
\newcommand\RMJ{\emph{Rocky Mt. J. Math.\ }}
\newcommand\SJNA{\emph{SIAM J. Numer. Anal.\ }}
\newcommand\SR{\emph{SIAM Rev.\ }}
\newcommand\TAMS{\emph{Trans. Amer. Math. Soc.\ }}
\newcommand\TOMS{\emph{ACM Trans. Math. Software\ }}
\newcommand\USSR{\emph{USSR Comput. Maths. Math. Phys.\ }}

\vskip 1pc
\noindent Mathematisches Seminar\\
Christian-Albrechts-Universit\"at zu Kiel\\
Ludewig-Meyn-Str.~4\\
24098 Kiel\\
Germany\\
{\em Email Address:} gnewuch@math.uni-kiel.de

\vskip 1pc
\noindent Fachbereich Mathematik\\
Technische Universit\"at Kaiserslautern\\
Postfach 3049\\
67653 Kaiserslautern\\
Germany\\
{\em Email Address:} hefter@mathematik.uni-kl.de

\vskip 1pc
\noindent Institut f\"ur Analysis\\
Johannes Kepler Universit\"at Linz\\
Altenberger Str.~69\\
4040 Linz\\
Austria\\
{\em Email Address:} aicke.hinrichs@jku.at

\vskip 1pc
\noindent Fachbereich Mathematik\\
Technische Universit\"at Kaiserslautern\\
Postfach 3049\\
67653 Kaiserslautern\\
Germany\\
{\em Email Address:} ritter@mathematik.uni-kl.de

\vskip 1pc
\noindent Department of Computer Science\\
Davis Marksbury Building\\
329 Rose St. \\
University of Kentucky\\
Lexington, KY 40506-0633, USA\\
{\em Email Address:} greg@cs.uky.edu

\begin{thebibliography}
\frenchspacing
\bibitem{Ai14}
C. Aistleitner,
Tractability results for the weighted star-discrepancy,
\JC {\bf 30} (2014), 381--391. 

\bibitem{Cre07}
J. Creutzig, Finite-diameter weights, manuscript, 2007. 

\bibitem{DSWW}
J. Dick, I. H. Sloan, X. Wang, and H. Wo\'zniakowski,
Good lattice rules in weighted Korobov spaces with general weights,
\NM {\bf 103} (2006), 63--97.

\bibitem{GnHeHiRi16}
M. Gnewuch, M. Hefter, A. Hinrichs, K. Ritter, 
Embeddings of Weighted Hilbert Spaces and Applications to 
Multivariate and Infinite-Dimensional Integration,
submitted, arXiv:1608.00906

\bibitem{HeRi13}
M. Hefter and K. Ritter, On embeddings of weighted tensor product
Hilbert spaces, 
\emph{J.~Complexity} {\bf 31} (2015), 405--423.

\bibitem{HeRiWa14}
M. Hefter, K. Ritter, and G. W. Wasilkowski, 
On equivalence of weighted anchored and ANOVA spaces
of functions with mixed smoothness of order one
in $L_1$ or $L_\infty$, 
\emph{J.~Complexity} {\bf 32} (2016), 1--19. 
DOI: 10.1016/j.jco.2015.07.001

\bibitem{HegWas}
M. Hegland and G. W. Wasilkowski, 
On tractability of approximation in special function spaces, 
\JC {\bf 29} (2013), 76--91. 
DOI: 10.1016/j/jco.2012.10.002. 

\bibitem{HS16} A. Hinrichs and J. Schneider, Equivalence of anchored
and {ANOVA} spaces via interpolation, \JC {\bf 33} (2016), 190--198.

\bibitem{HiPiSi08} A. Hinrichs, F. Pillichshammer, W. Ch. Schmid,
Tractability properties of the weighted star discrepancy, 
\JC {\bf 24} (2008), 134--143.

\bibitem{KriPilWas16} P. Kritzer, F. Pillichshammer, and G. W. Wasilkowski,
A note on equivalence of anchored and ANOVA spaces; 
lower bounds, 
\JC {\bf 35} (2016), 63--85.

\bibitem{KrPiWa16b}
P. Kritzer, F. Pillichshammer, and G. W. Wasilkowski,
Truncation dimnsion for approximation, {\em submitted}.

\bibitem{KSS}
F. Y. Kuo, Ch. Schwab, and I. H. Sloan (2012), 
Quasi-Monte Carlo finite element methods for a class of elliptic 
partial differential equations with random coefficients,
\SJNA {\bf 50}, 3351-3374.

\bibitem{KSWW09b}
F. Y. Kuo, I. H. Sloan, G. W. Wasilkowski, and H. Wo\'zniakowski,
On decompositions of multivariate functions, \MC {\bf 79} (2010), 
953--966.  DOI: 0.1090/S0025-5718-09-02319-9.

\bibitem{LeoPil03}
G. Leobacher, F. Pillichshammer,
Bounds for the weighted {$L^p$} discrepancy and tractability of 
integration,
\JC {\bf 19} (2003), 529--547. 

\bibitem{NowWoz08}
E. Novak and H. Wo\'zniakowski,
{\em Tractability of Multivariate Problems. Vol. I: Linear Information},
European Math. Society, Z\"urich, 2008. 

\bibitem{SloWoz}
I. H. Sloan and H. Wo\'zniakowski, 
When are quasi-Monte Carlo algorithms efficient for high
dimensional integrals? \JC {\bf 14} (1998), 1--33. 

\bibitem{Tr78} H. Triebel, Interpolation Theory, Function Spaces, 
Differential Operators. VEB Deutsch. Verl. Wissenschaften, Berlin 
1978, North-Holland, Amsterdam 1978.

\bibitem{Was14}
G. W. Wasilkowski, Tractability of approximation of $\infty$ variate
functions with bounded mixed partial derivatives, 
\JC {\bf 30} (2014), 325--346.

\end{thebibliography}
\end{document}